\documentclass[12pt]{amsart}
\usepackage{newcommands2006, amsfonts, amsmath, epic, amssymb, amsthm, bm}

\newtheorem{theorem}{Theorem}

\newtheorem{lemma}[theorem]{Lemma} 
\newtheorem{corollary}[theorem]{Corollary}
\theoremstyle{definition} 
\newtheorem{example}[theorem]{Example}
\theoremstyle{remark}

\newcommand{\bfRic}{\ensuremath{\mathbf{Ric}} }

\newcommand{\half}{\textstyle{\frac{1}{2}}  }

\begin{document}


\title{The Existence of Soliton Metrics for Nilpotent Lie Groups} 

\author{Tracy L. Payne} 

\address{ Department of Mathematics, Idaho State University, 
Pocatello, ID 83209-8085} 

\email{payntrac@isu.edu}


 
\begin{abstract} We show that a left-invariant metric $g$ on a
nilpotent Lie group $N$ is a soliton metric if and only if a matrix
$U$ and vector $v$ associated the manifold $(N,g)$ satisfy the matrix
equation $U v = \boldone,$ where $\boldone$ is a vector with every
entry a one.  We associate a generalized Cartan matrix to the matrix
$U$ and use the theory of Kac-Moody algebras to analyze the solution
spaces for such linear systems.  We use these methods to find
infinitely many  new examples of nilmanifolds with soliton metrics.
We give a sufficient condition for a sum of soliton metric nilpotent
Lie algebra structures to be soliton, and we use this criterion to
show that a soliton metric exists on every naturally graded filiform
metric Lie algebra. 
\end{abstract}

\maketitle


\setcounter{tocdepth}{2} 

\section{Introduction}
\label{intro}
\subsection{Overview} A {\em nilsoliton metric} is a left-invariant
metric on a nilpotent Lie group such that the Ricci endomorphism,
viewed as a linear mapping on the Lie algebra, differs from a scalar
multiple of the identity map by a derivation.  This condition may be
viewed as the Einstein condition at the second level of Lie algebra
cohomology (relative to the adjoint representation).

Nilsoliton metrics are of interest for several reasons.  First,
nonabelian nilpotent Lie groups do not admit left-invariant Einstein
metrics, and nilsoliton metrics have nice properties that make them
preferred metrics on nilpotent Lie groups in the absence of Einstein
metrics.  J.\ Lauret has shown that given a nilpotent Lie group, a
nilsoliton metric is unique up to scaling, and that nilsoliton metrics
arise as critical points of a natural Riemannian functional on the
space of normalized nilpotent Lie algebra brackets (\cite{lauret01a}).
Special metrics on nilpotent Lie groups, such as metrics from
horospheres of symmetric spaces of noncompact type, and more
generally, groups of type $H$, are nilsoliton metrics.  However,
nilsoliton metrics fail to have the desirable property of existence:
there exist many nilpotent Lie algebras that do not admit nontrivial
semisimple derivations and hence the corresponding nilpotent groups do
not admit nilsoliton metrics.

Second, soliton metrics arise in the study of Ricci flow, which was
introduced by Hamilton in \cite{hamilton82} as part of his program to
prove Thurston's Geometrization Conjecture for compact
three-manifolds.  Perelman has recently made great progress
(\cite{perelman1}, \cite{perelman2}, \cite{perelman3}) toward the
completion of that program.  Einstein metrics are the fixed points for
the normalized Ricci flow on the space of Riemannian metrics on a
compact manifold.  A solution to the Ricci flow which moves by
diffeomorphisms and also dilates by a time-dependent factor is called
a {\em Ricci soliton.}  Ricci soliton metrics may be viewed as
generalized fixed points for the Ricci flow.  A left-invariant metric
on a nilpotent Lie group is a Ricci soliton metric if and only if it
is a nilsoliton metric (Prop. 1.1, \cite{lauret01a}).
\nocite{chowknopf04}
 
Third, nilsoliton metrics are inherent in the study of Einstein
metrics on noncompact homogeneous spaces.  All known noncompact,
nonflat, homogeneous Einstein manifolds are isometric to solvable Lie
groups endowed with left-invariant metrics (solvmanifolds).  D.\ V.\
Alekse'evskii conjectured that an isotropy subgroup of the isometry
group of a noncompact, nonflat, homogeneous Einstein manifold is a
maximal compact subgroup. If Alekse'evskii's conjecture is true, then
any noncompact, nonflat, homogeneous Einstein manifold with linear
isometry group would be isometric to a solvmanifold.

It has been shown that left-invariant Einstein metrics on unimodular
solvable groups are flat (\cite{dottimiatello82}). 
The remaining case of interest is when the solvable group is
nonunimodular.  In \cite{heberinv}, J.\ Heber gave a comprehensive
analysis of structure of nonunimodular Einstein solvmanifolds and
their moduli spaces.  We say that a metric Lie algebra is Einstein (or
nilsoliton) if the corresponding simply connected homogeneous space is
Einstein (or nilsoliton).  A metric solvable Lie algebra $\fraks$ is
called {\em standard} if $\fraka := [\fraks, \fraks]^\perp$ is
abelian. All Einstein solvable metric Lie algebras are of standard
type (\cite{lauret07}).   Heber showed that any nonunimodular solvable
Lie algebra $\fraks$ endowed with a standard Einstein metric $Q$ has a
subalgebra $\fraks_0 := \la H \ra \oplus \frakn$ so that $(\fraks_0,
Q|_{\fraks_0})$ is Einstein. Thus, the analysis of standard
nonunimodular Einstein solvmanifolds is reduced to the case where
$\fraka$ has dimension one (the {\em rank one} case).

In \cite{lauret01a} and \cite{lauret01b}, building on the work of
Heber, J.\ Lauret established a correspondence between rank one
Einstein metric solvable Lie algebras and nilsoliton metric Lie
algebras.  A metric solvable extension of a metric nilpotent Lie
algebra $(\frakn, Q)$ with Lie bracket $[\cdot, \cdot]_\frakn$ is a
solvable Lie algebra $\fraks = \fraka \oplus \frakn$ with Lie bracket
$[\cdot, \cdot]_\fraks$ and inner product $\tilde Q$ such that
$[\fraks,\fraks]_\fraks = \frakn = \fraka^\perp,$
$[\cdot,\cdot]_\fraks|_{\frakn} = [\cdot, \cdot ]_\frakn,$ and $\tilde
Q|_{\frakn \times \frakn} = Q.$ 
 Lauret
proved that a metric nilpotent Lie algebra $(\frakn,Q)$ is nilsoliton
if and only if $(\frakn,Q)$ admits a metric solvable extension
$(\fraks = \fraka \oplus \frakn, \tilde Q)$ such that $\fraka$ is
abelian and the corresponding solvmanifold $(S,\tilde g)$ is Einstein.
(Theorem 3.7 of \cite{lauret01b}).  This theorem may be used to
translate theorems in this paper about nilsoliton metrics to parallel
statements about Einstein solvmanifolds.  Given a nilsoliton metric
Lie algebra, it is the semisimple derivation $D,$ the difference of
the Ricci endomorphism and the scalar multiple of the identity map,
that defines the Einstein solvable extension: that is, $\fraks :=
\boldR H \oplus \frakn$ and $[\cdot,\cdot]_\frakn$ is extended to
$[\cdot,\cdot]_\fraks$ by setting $[H,X]_\fraks := D X$ for $X \in
\frakn.$ Nonabelian nilpotent Lie algebras do not admit Einstein
metrics (\cite{jensen-69}, see also Theorem \ref{niceformula}), so the
derivation $D$ for a nilsoliton metric Lie algebra is trivial if and
only if $\frakn$ is abelian.

In this work, we study nilsoliton metrics using algebraic and
combinatorial methods.  We work in the class of metric (nonassociative
skew) algebras.  A {\em metric algebra} $(\frakn_\mu,Q)$  is a vector
space $\frakn$ equipped with an inner product $Q$ and an algebraic
structure defined by an element $\mu: \frakn \times \frakn \to \frakn$
in the vector space $\Lambda^2 \frakn^\ast \otimes \frakn$ of
skew-symmetric bilinear vector-valued maps.   The expression for the
Ricci form for a nilmanifold extends to define a  ``nil-Ricci''
endomorphism $\Ric_\mu$ for each metric algebra.  Our main interest is
in nilpotent algebras.  If $\Ric_\mu - \beta \Id$ is a derivation of
$\frakn_\mu$ we say that $(\frakn_\mu,Q)$ satisfies the {\em
nilsoliton condition} with {\em nilsoliton constant} $\beta.$

We associate to each nonabelian metric algebra $(\frakn_\mu,Q)$ and
orthonormal basis $\calB$ of eigenvectors for the nil-Ricci
endomorphism an array of combinatorial objects: a set $\Lambda$ of
integer triples encoding the nontrivial structure constants for the
basis, a set of vectors with entries of zero, one and minus one
indexed by $\Lambda,$ a Gram matrix $U$ encoding the inner products of
these vectors, and graph $S(U)$ defined by the matrix $U.$ We will see
that interesting algebraic and geometric properties of
$(\frakn_\mu,Q)$ and $\calB$ are displayed in properties of these
combinatorial objects.

We associate to $(\frakn_\mu,Q)$ a vector $[\alpha^2]$ listing the
squares of the nontrivial structure constants.  It is shown that the
nil-Ricci endomorphism satisfies the nilsoliton condition if and only
if the vector $[\alpha^2]$ is a solution to the matrix equation $U v =
\boldone$ where $\boldone$ is a column vector with every entry a one,
thus reducing the difficult tensorial problem of finding Einstein and
soliton metrics to a problem of linear algebra.  We define a
generalized Cartan matrix in terms of the Gram matrix $U$ and apply
results from the theory of Kac-Moody algebras to analyze the
existence, uniqueness, and properties of solutions to the special
systems $Uv=\boldone.$ From these results we get a procedure for
finding all nilsoliton metrics in a fixed dimension, and hence all
rank one Einstein solvmanifolds of any given dimension.  Once we have
a  rank one Einstein solvmanifold it is not hard to find all higher
rank extensions by computing  the derivation algebra of the
nilradical.

We study the linear systems arising from some examples, including
nilpotent Lie groups associated to rank one symmetric spaces of
noncompact type, and families of filiform nilpotent Lie groups.

 When elements $\mu_1$ and $\mu_2$ in $\Lambda^2 \frakn^\ast \otimes
\frakn$ define algebras $\frakn_{\mu_1}$ and $\frakn_{\mu_2}$ on the
inner product space $(\frakn,Q)$ so that the metric algebras
$(\frakn_{\mu_1},Q)$ and $(\frakn_{\mu_2},Q)$ satisfy the nilsoliton
condition, in general the sum $\mu_1 + \mu_2$ does not define an
algebra $\frakn_{\mu_1 + \mu_2}$ so that $(\frakn_{\mu_1 + \mu_2},Q)$
again satisfies the nilsoliton condition; a sufficient condition to
guarantee this is found.  This criterion is useful for constructing
nilsoliton metrics, and for showing in certain situations that a
metric Lie algebra of interest is nilsoliton.

\subsection{Summary of Results} A {\em nilmanifold} is a connected
Riemannian manifold with a transitive nilpotent group of isometries.
A nilmanifold is isometric to a nilpotent Lie group $N$ endowed with a
left-invariant Riemannian metric $g.$ P.\ Eberlein and others have
studied the geometry of two-step nilmanifolds in depth (see
\cite{eberlein04} for a survey and references), but little has been
done in the general higher-step case, where the expressions for the
Levi-Civita connection and the exponential map are markedly more
complicated than in the two-step case.

 We associate to the nilmanifold $(N,g)$ the metric Lie algebra
$(\frakn,Q),$ where $\frakn$ is the Lie algebra of $N$ and $Q$ is the
restriction of the metric $g$ on $N$ to the tangent space at the
identity $T_eN \cong \frakn.$ As we are always interested in purely
local properties, we identify $(N,g)$ with $(\frakn,Q),$ and when we
refer to things like the connection, curvature, etc.\ for the metric
Lie algebra $(\frakn,Q),$ we mean the structures associated to
$(\frakn,Q)$ coming from the connection, curvature, etc.\ for the
homogeneous space $(N,g).$

Although our primary aim is to investigate nilsoliton metrics for
metric nilpotent Lie algebras, it is convenient to generalize the
setting and allow $(\frakn_\mu,Q)$ to be a metric  algebra, where $\mu
= [\cdot, \cdot]$ is in $\Lambda^2 \frakn^\ast \otimes \frakn.$ Let
$\calB = \{X_i\}_{i=1}^n$ be a $Q$-orthonormal basis of $\frakn_\mu.$
(We always assume that bases are ordered.) The {\em nil-Ricci
endomorphism} $\Ric_\mu$ is defined as $\la \Ric_\mu X, Y \ra =
\ric_\mu(X,Y),$ where
\begin{equation}\label{nilriccicurvature} \ric_\mu (X,Y) =
-\frac{1}{2}\sum_{i=1}^n\la [X,X_i], [Y,X_i]\ra+\frac{1}{4} \sum_{i,j
= 1}^n\la [X_i,X_j], X\ra \la [X_i,X_j], Y\ra
\end{equation} for $X, Y \in \frakn_{\mu}.$ (We often write an inner
product $Q(\cdot,\cdot)$ as $\la \cdot, \cdot \ra.$)  When $\frakn$ is
a nilpotent metric Lie algebra, the nil-Ricci endomorphism is the
Ricci endomorphism.   If all elements of the basis are eigenvectors
for the nil-Ricci endomorphism $\Ric_\mu,$ we call the orthonormal
basis a {\em Ricci eigenvector basis}.  By the symmetry of the
nil-Ricci form, a Ricci eigenvector basis always exists.

Now we define some combinatorial objects associated to a set of
integer triples $\Lambda \subset \{ (i,j,k) \, | \, 1 \le i,j,k \le n
\}.$ We will later consider subsets of the following sets:
\begin{align*} \Upsilon_n &= \{ (i,j,k) \, | \, 1 \le i < j \le n,
1\le k \le n\} \\ \Psi_n &= \{ (i,j,k) \, | \, 1 \le i < j \le n, 1
\le k \le n, i \ne k, j \ne k \} \\ \Theta_n &= \{ (i,j,k)\} \, | \, 1
\le i < j < k \le n \}
\end{align*} For $1 \le i,j,k \le n,$ define the $1 \times n$ row
vector $Y_{ij}^k$ to be $\varepsilon_i^T + \varepsilon_j^T -
\varepsilon_k^T,$ where $\{ \varepsilon_i\}_{i = 1}^n$ is the standard
orthonormal basis for $\boldR^n.$ We call the vectors in $\{Y_{ij}^k
\, | \, (i,j,k) \in \Lambda \}$ {\em root vectors} for $\Lambda.$ Let
$Y_1,Y_2, \ldots, Y_m$ (where $m = \| \Lambda \|$) be an enumeration
of the root vectors in dictionary order.  We define the {\em root
matrix} $Y_\Lambda$ for $\Lambda$ to be the $m \times n$ matrix whose
rows are the root vectors $Y_1, Y_2, \cdots, Y_m.$ The {\em Gram
matrix} $U_\Lambda$ {\em for} $\Lambda$ is the $m \times m$ matrix
defined by $U_\Lambda = Y_\Lambda Y_\Lambda^T;$ the $(i,j)$ entry of
$U_\Lambda$ is the inner product of the $i$th and $j$th root vectors.
Because there are a finite number of root vectors in any dimension,
there are only finitely many possible root matrices and Gram matrices
in any dimension.

The Gram matrix $U = (u_{ij})$ defines a graph $S(U),$ which we will
call the {\em generalized Dynkin diagram} of $U.$  The vertices of the
generalized Dynkin diagram are in one-to-one correspondence with
triples $(i,j,k)$ in $\Lambda.$ Let these vertices be named $n_1, n_2,
\ldots n_m$ when listed in dictionary order.  Between vertices $n_i$
and $n_j$ where $i \ne j$ and $u_{ij} \ne 0$, draw an edge and label
it with $u_{ij}.$ If the label is one, the label is omitted and if the
label is minus one, the edge may be drawn as a dashed line segment.
If $u_{ij}=0,$ the vertices $n_i$ and $n_j$ are not connected by an
edge.  Technically, $S(U)$ is not a generalized Dynkin diagram as
defined in \cite{kac90} since $U$ is not a generalized Cartan matrix.
There is a true generalized Cartan matrix $A$ associated to $U$ when
$\Lambda \subset \Psi_n.$ Its edge-weighted graph $S(A)$ is closely
related to the graph $S(U);$ however, we work work with $S(U)$ as $U$
is more natural than $A$ from a geometric point of view, and $U$ is
uniquely defined, whereas there are different good choices for the
matrix $A$ depending on the situation.

Next we define some structures associated to a nonabelian
$n$-dimensional metric algebra $(\frakn_\mu,Q)$ and an orthonormal
basis $\calB$ for $(\frakn_\mu,Q).$ Let $\alpha_{ij}^k$ denote the
structure constant $\la [X_i,X_j], X_k \ra.$ The set
$\Lambda(\frakn_\mu,\calB) = \{ (i,j,k) \, | \, i < j, \alpha_{ij}^k
\ne 0\}$ indexes the set of nonzero structure constants, ignoring
repetitions due to skew-symmetry. The set of triples
$\Lambda(\frakn_\mu,\calB)$ is a subset of $\Upsilon_n.$ We may say
that the root vectors, root matrix, and Gram matrix for $\Lambda =
\Lambda(\frakn_\mu,Q)$ are {\em for $(\frakn_\mu,Q)$ with respect to
the basis $\calB.$} Let $[\alpha^2]_{(\frakn_\mu,\calB)}$ denote the
$m \times 1$ vector with the entries $(\alpha_{ij}^k)^2$ for $(i,j,k)
\in \Lambda(\frakn_\mu, \calB),$ again taken in dictionary order.  We
call $[\alpha^2]_{(\frakn_\mu,\calB)}$ the {\em structure vector for
$(\frakn_\mu, Q)$ with respect to $\calB.$} The structure vector
$[\alpha^2]$ defines a weighting $w$ of the vertices of $S(U),$
assigning the $i$th vertex $n_i$ the $i$th entry of $[\alpha^2],$ so
that the vertex for the triple $(i,j,k)$ has the weight
$(\alpha_{ij}^k)^2.$ We will sometimes abbreviate our notation, for
example writing $\Lambda$ for $\Lambda(\frakn_\mu,\calB)$ or omitting
subscripts, when it is clear what is meant.  A single nilmanifold may
have different root vectors, structure vectors, and Gram matrices,
depending on the choice of Ricci eigenvector basis.  Nonisometric
nilmanifolds may have the same Gram matrices, structure vectors, and
Dynkin diagrams.

In Section \ref{curvature}, we give simple formulas for computing
Ricci and sectional curvature for metric Lie algebras relative to
Ricci eigenvector bases, and we illustrate these formulas with
examples.

In Section \ref{properties}, we show that a metric algebra
$(\frakn_\mu,Q)$ satisfies the nilsoliton condition if and only if the
structure vector with respect to $(\frakn_\mu,Q)$ and any Ricci
eigenvector basis $\calB$ is a solution to a linear system associated
to $(\frakn_\mu,Q).$ We use $[a]_{m \times n}$ to denote the $m \times
n$ matrix all of whose entries are the real number $a.$
\begin{theorem}\label{mainthm} Let $(\frakn_\mu,Q)$ be a nonabelian
metric algebra with Ricci eigenvector basis $\calB.$ Let $U$ and
$[\alpha^2]$ be the Gram matrix and the structure vector for
$(\frakn_\mu,Q)$ with respect to $\calB.$ Then $(\frakn_\mu,Q)$
satisfies the nilsoliton condition with nilsoliton constant $\beta$ if
and only if $U [\alpha^2] = -2\beta \onevector{m}.$
\end{theorem} Thus, to find all nilsoliton metrics, one needs to find
all vectors $v$ with positive entries that are solutions to systems of
the form $Uv=-2\beta \onevector{m},$ where $U$ is the Gram matrix
defined by a set of root vectors.  Choosing structure constants whose
squares are the entries of $v$ will define a metric algebra
$(\frakn_\mu,Q)$ satisfying the nilsoliton condition.  The sets of
integer triples are enumerable, so this gives an algorithm for finding
nilsolitons.

A solution to the equation $Uv=-2\beta\onevector{m}$ gives a weighting
$w$ on the vertices of the Dynkin diagram $S(U)$ with the property
that for all $i,$
\begin{equation}\label{weighting} 3 w(n_i) + \sum_{ u_{ij} \ne 0}
u_{ij} w(n_j) = -2\beta ;
\end{equation} that is, thrice  the value of the weighting at a vertex
$n_i$ plus the edge-weighted sum of the values at adjacent vertices is
constant.  This interpretation of solutions allows one to use the
symmetries of the generalized Dynkin diagram $S(U)$ to find symmetries
in the solution space of the system $Uv=-2\beta \onevector{m}.$

Elementary properties of the Gram matrices  for metric nilpotent Lie
algebras are derived in Section \ref{properties}.  In Section
\ref{jacobi}, Theorem \ref{jacobicondition} interprets the Jacobi
condition for a metric nilpotent algebra $\frakn$ in terms of the
weighted graph $S(U_\Lambda),$ so that one may check whether $\frakn$
is a Lie algebra by looking at the graph $S(U_\Lambda).$ The proof of
Theorem \ref{mainthm} is in Section \ref{proof of main thm}, along
with some theorems describing properties of derivations and gradings
of nilpotent metric algebras that are needed for the proof.  Some of
the results in Section \ref{derivations} about semisimple derivations
of nilpotent algebras may be of independent interest.

In Section \ref{solutions} the solution spaces to linear systems of
the form $Uv=-2\beta\onevector{m}$ are analyzed.  It is shown in
Theorem \ref{gramproperties} that if a metric algebra $(\frakn_\mu,Q)$
with Ricci eigenvector basis $\calB$ has the property that
$\alpha_{ij}^j = 0$ for all $i$ and $j,$ then the Gram matrix $U$ can
be written in the form $U = cA + d[1]_{m \times m}$ where $A$ is a
generalized Cartan matrix and $c$ and $d$ are positive.  We call $A$
an {\em associated matrix} for $U.$ Theorem \ref{solsp} characterizes
the solution spaces for the linear systems $Uv=(cA + d[1]_{m \times
m})v=-2\beta\onevector{m}$ by the properties of a generalized Cartan
matrix $A$ associated to $U.$ We recall that  an irreducible
generalized Cartan matrix is one of three types: finite type, affine
type, or indefinite type.  For a column vector $u = (u_1, u_2,
\ldots,)^T$ we write $u > 0$ ($u \ge 0$) to denote that $u_i > 0$
($u_i \ge 0$) for all $i.$
\begin{theorem}\label{solsp} 
Let $A$ be a $m \times m$ irreducible generalized Cartan matrix and
let $U = cA + d \onematrix{m}{m},$ where $c$ and $d$ are positive.
Suppose $\beta < 0.$ Then there exists a unique real number $\mu$ such
that solutions to
\begin{equation}
\label{above} 
Uv = -2\beta \onevector{m}
\end{equation}
are the same as solutions to
\begin{equation}
Av = \mu \onevector{m}.
\end{equation}
There are five mutually exclusive possibilities for the common
solution space:
\begin{enumerate}
\item[(Fin)]{$A$ is of finite type, $\mu > 0,$ Equation \eqref{above}
has a unique solution $v,$ and $v > 0.$ }
\item[(Aff)]{$A$ is of affine type, $\mu = 0,$ Equation \eqref{above}
has a unique solution $v,$ and $v > 0.$ }
\item[(Ind1)]{$A$ is of indefinite type, $\mu < 0,$ Equation
\eqref{above} has a positive solution $v_0,$ the general solution to
Equation \eqref{above} is
\[ v_0 + \ker A = v_0 + \ker U = v_0 + \ker Y^T,\] and the 
intersection of the solution space and the cone $\{ v \in \boldR^m \,
| \, v \ge 0 \}$ is a simplex of dimension $\ker A$ parallel to
$\onevector{m}^\perp.$}
\item[(Ind2)]{$A$ is of indefinite type, $\mu < 0$ and Equation
\eqref{above} is consistent but has no solution $v$ with $v > 0.$}
\item[(Ind3)]{$A$ is of indefinite type, $\mu < 0,$ and Equation
\eqref{above} is not consistent.}
\end{enumerate}
Additionally, if $U$ is the Gram matrix for a metric algebra
$(\frakn_\mu,Q)$ with respect to a Ricci eigenvector basis $\calB,$
then Case (Ind3) can not occur.
\end{theorem}
Each of the first four cases actually occurs in systems arising from
metric nilpotent Lie algebras; we present examples of each type.

In Section \ref{proof of allnilsolitons} we consider parameter spaces
of metric nilpotent Lie algebras.  Before we state the main result of
the section, we define some terminology.  Fix an orthonormal basis
$\calB$ of an $n$-dimensional inner product space $(\frakn,Q).$ Let
$\calR(\frakn,\calB)$ be the subset of $\Lambda^2 \frakn^\ast \otimes
\frakn$ consisting of all $\mu$ so that $\calB$ is a Ricci eigenvector
basis for $(\frakn_\mu,Q).$ Every $n$-dimensional metric algebra is
metrically isomorphic to $(\frakn_\mu,Q)$ for some $\mu$ in
$\calR(\frakn,\calB)$ because Ricci eigenvector bases always exist.

We let $\calT(\frakn,\calB)$ be the set of all nontrivial elements
  $\mu$ of that $\Lambda^2 \frakn^\ast \otimes \frakn$ have the
  property that $\Lambda(\frakn_\mu,\calB)$ is a subset of $\Theta_n.$
  We will say that an element of $\calT(\frakn,\calB)$ is $\calB$-{\em
  triangular}.  Note that for all $\mu$ in $\calT(\frakn,\calB),$ the
  algebra $\frakn_\mu$ is nilpotent.  Every $n$-dimensional metric
  algebra satisfying the nilsoliton condition is represented by an
  element of $\calR(\frakn,\calB) \cap \calT(\frakn,\calB)$ (by
  Corollary \ref{goodbasis} in Section \ref{geometry}).

Let $\calL(\frakn,\calB)$ be the subset of $\Lambda^2 \frakn^\ast
\otimes \frakn$ of Lie algebras: the algebraic variety consisting of
those $\mu$ so that $\frakn_\mu$ satisfies the Jacobi identity.  The
general linear group $GL_n(\boldR)$ acts on $\Lambda^2 \frakn^\ast
\otimes \frakn$ by $(g \cdot \mu) (X,Y) = g \mu(g^{-1}X,g^{-1}Y).$ Two
Lie brackets $\mu$ in $\calL(\frakn,\calB)$ define isomorphic Lie
algebras if they are in the same $GL_n(\boldR)$ orbit.  Hence the
space of $n$-dimensional Lie algebras can be identified with quotient
space of $\calL(\frakn,\calB)$ under this $GL_n(\boldR)$ action.

Fix an orthonormal basis for the metric algebra $(\frakn_\mu,Q).$ The
set $\Lambda^2 \frakn^\ast \otimes \frakn$ is partitioned into
equivalence classes $\Omega_\Lambda$ that are level sets for the
function assigning an element $\mu$ to the set
$\Lambda(\frakn_\mu,\calB)$ of nonzero structure constants for
$(\frakn_\mu,Q)$ with respect to the basis $\calB.$ When $\Lambda =
\emptyset,$ we call $\Omega_\Lambda$ {\em trivial}.  For each element
$\mu$ in a nontrivial equivalence class $\Omega_\Lambda,$ the root
matrix $Y$ and Gram matrix $U$ for $(\frakn_\mu,Q)$ with respect to
$\calB$ are the same, as they depend only on the common set
$\Lambda(\frakn_\mu,\calB).$ This partition restricts to a partition
of $\calT(\frakn,\calB)$ because we can write $\calT(\frakn,\calB)$ as
the union $\cup_{\varnothing \ne \Lambda \subset \Theta_n }
\Omega_\Lambda$ of equivalence classes $\Omega_\Lambda$ for all
nontrivial subsets $\Lambda$ of $\Theta_n.$

The description of the set of metric algebras with a specified set of
nonzero structure constants that satisfy the nilsoliton condition is
given in the next theorem.
\begin{theorem}\label{allnilsolitons}
Let $(\frakn,Q)$ be an inner product space with orthonormal basis
 $\calB.$ Let $\Omega_\Lambda \subset \Psi_n.$ Let $U$ denote the Gram
 matrix for elements of $\Omega_\Lambda$ and let $m = \| \Lambda \|.$
 Suppose $\beta < 0.$ Then the set of metric algebras $(\frakn_\mu,Q)$
 in $\Omega_\Lambda$ whose structure vector satisfies $U [\alpha^2] =
 -2 \beta \boldone$ is
\begin{itemize}
\item{finite of size $2^m$, if $U$ is nonsingular and there is a
unique positive solution to $U v = -2\beta \boldone,$ }
\item{a finite disjoint union of the interiors of $2^m$ simplices of
dimension $\ker U$ if $U$ is singular and $U v = -2\beta \boldone$ has
a solution $v > 0,$ or }
\item{empty if there are no solutions to $U v = -2\beta \boldone$ with
$v > 0.$ }
\end{itemize}
Also, the endomorphism $\Ric_\mu$ is the same for all elements
$\mu$ of $\Omega_\Lambda \cap \calR(\frakn,\calB)$ such that $U
[\alpha^2]_\mu = -2 \beta \boldone.$
\end{theorem}

In Section \ref{examples}, we present examples of nilsoliton metrics
constructed by solving equations of the form $Uv = \boldone.$  The
linear systems associated to the higher rank symmetric space $\calP_n
= SL_n(\boldR)/SO(n),$ the quaternionic hyperbolic space $H^2(\boldH)$
and the Cayley plane are computed and solved, yielding continuous
families of Einstein solvmanifolds around the latter two symmetric
spaces and around $\calP_n$ for $n \ge 5.$
 Linear systems $U v =
\boldone$ for the family of the simplest
 filiform Lie algebras,  those
with a codimension one abelian ideal, are computed and used to prove
in Theorem \ref{Ln} the existence of nilsoliton metrics unique up to
scaling on these Lie algebras in each dimension, adding to the
analytic proof of this theorem by J.\ Lauret from \cite{lauret02} a
new algebraic proof.  
 We address the question of which metric
nilpotent Lie algebras admit symmetric derivations in Section
\ref{derivations}, and state a necessary combinatorial condition.  In
Theorem \ref{R6} we give an example of a metric nilpotent Lie algebra
that admits a positive definite derivation but does not admit a
nilsoliton metric; this shows that characteristic nilpotence is not
the only obstruction to the existence of a a nilsoliton metric on a
nilpotent Lie algebra.

In Section \ref{adding}, we address the question of
 when the sum of two nilsoliton
structures on an inner product space is nilsoliton.  For a symmetric
map $D$ of an $n$-dimensional vector space with a basis of
$D$-eigenvectors $\calB=\{X_i\}_{i=1}^n,$ we define the {\em eigenvalue vector
$v_D$ of $D$ with respect to $\calB$} to be $ v_D = (\mu_1, \ldots,
\mu_n)^T,$ where $\mu_i$ is the eigenvalue of $X_i$ for $i = 1,
\ldots, n.$ For a $n$-dimensional metric algebra $(\frakn_\mu,Q)$ 
with orthonormal basis $\calB,$ define the {\em Ricci vector} to be
\begin{equation}\label{riccivector} 
\bfRic^\calB_\mu = (\ric_\mu(X_1,X_1), \ric_\mu(X_2,X_2), \ldots,
\ric_\mu(X_n,X_n))^T .\end{equation} When $\calB$ is an
Ricci eigenvector basis, 
the Ricci vector is the eigenvalue vector
$v_{\Ric}$ for the Ricci endomorphism (recall that Ricci eigenvector bases 
are always assumed to be orthonormal).  We establish a sufficient
condition, formulated in terms of Ricci vectors, for the sum of metric
algebraic structures $\mu_1$ and $\mu_2$ on an inner product space
$(\frakn,Q)$ to satisfy the nilsoliton condition:
\begin{theorem}\label{addingnilsolitons} 
Let $(\frakn,Q)$ be an inner product space with orthonormal basis
$\calB$.  Let $\mu_1$ and $\mu_2$ be nontrivial elements of
$\calR(\frakn, \calB)$ defining metric algebras $(\frakn_{\mu_1},Q)$
and $(\frakn_{\mu_2},Q)$ satisfying the nilsoliton condition with
nilsoliton constants $\beta_1$ and $\beta_2$ respectively.  Suppose
that the set $ \Lambda(\frakn_{\mu_1 + \mu_2},\calB) $ is the disjoint
union of the sets $\Lambda(\frakn_{\mu_1},\calB)$ and
$\Lambda(\frakn_{\mu_2},\calB).$ If there are constants $c_1$ and
$c_2$ so that $ Y_{ij}^k \, \bfRic_{\mu_1}^\calB = c_1$ for all
$(i,j,k)$ in $\Lambda(\frakn_{\mu_2}, \calB),$ $Y_{ij}^k \,
\bfRic_{\mu_2}^\calB = c_2$ for all $(i,j,k)$ in
$\Lambda(\frakn_{\mu_1}, \calB),$ and $-2\beta_1 + c_2= -2\beta_2 +
c_1 =: \beta,$ then $(\frakn_{\mu_1 + \mu_2}, Q)$ satisfies the
nilsoliton condition with nilsoliton constant $\beta.$
\end{theorem}
We use this theorem in the proof of Theorem \ref{Qn} which asserts the
existence of nilsoliton metrics on each member of an infinite family
$Q_n$ of filiform metric nilpotent Lie algebras.  Combining the
existence results in Theorems \ref{Ln} and \ref{Qn}, we conclude in
Corollary \ref{naturallygraded} that every naturally graded filiform
nilpotent Lie algebra admits a nilsoliton metric.

\section{Preliminaries}
\subsection{The geometry of metric nilpotent Lie algebras}\label{geometry}
Let $(N,g)$ be a nilmanifold with associated metric Lie algebra
$(\frakn, Q).$ We will denote by $X$ both a vector in $\frakn$ and the
left-invariant vector field on $N$ that it induces.  The connection
for $(N,g)$ is given by
\begin{equation}\label{connection}
 \nabla_X Y = \half \left( \ad_X Y - \ad^\ast_X Y - \ad^\ast_Y X
\right),
\end{equation} 
and the sectional curvature $K$ of a plane spanned by orthonormal $X$
and $Y$ is
\begin{equation}\label{K} 
K(X \wedge Y) = \|\nabla_X Y\|^2 - \la \nabla_X X , \nabla_ Y Y \ra -
\la \ad_Y^2 X, X \ra - \|\ad_X Y\|^2 \end{equation} (\cite{wolter91a},
\cite{alekseevski75a}).  Let $\{X_i\}$ be an orthonormal basis of
$\frakn.$ We view the Ricci form for $(N,g)$ as in inner product on
$\frakn$ and the Ricci endomorphism as an linear map from $\frakn$ to
itself. The Ricci form $\ric$ for $(N, g)$ is given by the expression
for the nil-Ricci form in Equation \eqref{nilriccicurvature}, where we
denote the metric $g$ by $\la \cdot , \cdot \ra.$ This is a special
case of the the general formula for the Ricci form for a homogeneous
space (Corollary 7.38, \cite{besse-einsteinmanifolds}).

\subsection{Generalized Cartan matrices and their Dynkin diagrams}
The following background material is from Chapter 4 of \cite{kac90},
where it is used to study root systems of Kac-Moody algebras,
generalizing classical results for semisimple Lie algebras.

A real $n \times n$ matrix $A = (a_{ij})_{i,j=1}^n$ is called a {\em
generalized Cartan matrix} if
\begin{enumerate}
\item{$a_{ii}=2$ for $i=1, \ldots, n,$}
\item{$a_{ij}$ are nonpositive integers for $i \ne j,$ \text{and}}
\item{if $a_{ij} = 0,$ then $a_{ji}=0$}
\end{enumerate}
A matrix is {\em decomposable} if it is block diagonal after
conjugation by a permutation matrix, and is {\em indecomposable} if it
is not decomposable.  A vector $v$ will be called {\em positive} if $v
> 0.$

\begin{theorem} [\cite{kac90}]\label{kac4-3}
  Let $A$ be a real $n \times n$ indecomposable matrix so that
$a_{ij}$ is nonpositive if $i \ne j$ and $a_{ij}=0$ implies that
$a_{ji}=0.$ Then one and only one of the following three possibilities
holds for both $A$ and $A^T:$
\begin{enumerate}
\item[(Fin)]{$\det A \ne 0;$ there exists $u > 0$ such that $Au > 0;$
and if $Av \ge 0$ then $v > 0$ or $v = 0.$}
\item[(Aff)]{$\dim \ker A = 1;$ there exists $u > 0$ such that $Au =
0;$ and if $Av \ge 0$ then $Av = 0.$}
\item[(Ind)]{there exists $u > 0$ such that $Au < 0;$ and if $Av \ge
0$ and $v \ge 0,$ then $v=0.$ }
\end{enumerate}
\end{theorem}

These three cases are called {\em finite type, affine type} and {\em
indefinite type} respectively.  The Dynkin diagrams of all generalized
Cartan matrices of finite and affine type are classified.

\section{Curvatures for metric nilpotent algebras}\label{curvature}
\subsection{Ricci curvature}
We often want to represent the nil-Ricci endomorphism for a given
metric algebra $(\frakn,Q)$ with respect to a basis.  Corollary
\ref{orthogonal} allows us to check whether a basis is a Ricci
eigenvector basis.  Once we have a Ricci eigenvector basis, Theorem
\ref{riccitensor} will give the eigenvalues for each of the basis
vectors, thus determining the endomorphism.

For a metric algebra $(\frakn_\mu,Q),$ the linear map $J: \frakn \to
\End(\frakn)$ is   defined by $J_X Y=\ad_Y^\ast X;$ it encodes the
algebraic structure of $\frakn_\mu$ using the inner product $Q$.  The
inner product $\la \cdot , \cdot \ra$ on $\frakn$ induces an inner
product on the tensor algebra of $\frakn,$ which we also denote by 
$\la \cdot , \cdot \ra.$ In the next theorem, we use this inner
product to write the Ricci form for
 $(\frakn_\mu,Q)$  in terms of the adjoint and $J$
maps in $\End(\frakn).$

\begin{theorem}\label{niceformula}  
Suppose that $(\frakn_\mu,Q)$ is a nilpotent metric algebra. Then
  \[ \ric_\mu(X,Y) =  -\frac{1}{2} \la \ad_X ,  \ad_Y  \ra + \frac{1}{4} \la
  J_X , J_Y \ra \] for any $X, Y \in \frakn.$
\end{theorem}

\begin{proof}
Let $\{X_i\}_{i=1}^n$ be an orthonormal basis for $(\frakn_\mu,Q).$
From Equation \eqref{nilriccicurvature} we get
\begin{align*} 
\ric_\mu&(X,Y) \\ &= -\frac{1}{2} \sum_{i=1}^n \la \ad_X X_i, \ad_Y
X_i \ra + \frac{1}{4} \sum_{i,j=1}^n \la \ad_{X_i} X_j, X \ra \la
\ad_{X_i} X_j, Y \ra \\ &= -\frac{1}{2}\sum_{i,j=1}^n \la \ad_X X_i,
X_j \ra \la \ad_Y X_i, X_j \ra \\ & \hskip 1in + \frac{1}{4}
\sum_{i,j=1}^n \la X_j,\ad_{X_i}^\ast X \ra \la X_j,\ad_{X_i}^\ast Y
\ra \\ &= -\frac{1}{2} \la \ad_X , \ad_Y \ra + \frac{1}{4}
\sum_{i,j=1}^n \la X_j,J_X X_i \ra \la X_j,J_Y X_i \ra \\ &=
-\frac{1}{2} \la \ad_X , \ad_Y \ra + \frac{1}{4} \la J_X , J_Y \ra
\end{align*}
 \end{proof}

This has a corollary that will be useful in concrete situations.
\begin{corollary}\label{orthogonal} Let
$(\frakn_\mu,Q)$ be an $n$-dimensional metric algebra.  Then an orthonormal
basis  $\{
X_i\}_{i=1}^n$ 
is a Ricci eigenvector basis if $\{J_{X_i}\}_{i=1}^n$
and $\{\ad_{X_i}\}_{i=1}^n$ are both orthogonal sets in
$\End(\frakn)$.
\end{corollary}

If $\{J_{X_i}\}_{i=1}^n$
$\{\ad_{X_i}\}_{i=1}^n$ are both orthogonal sets in $\End(\frakn)$,
then $\calB$ is a Ricci eigenvector basis by the corollary above.
The orthogonality of these sets can sometimes be verified simply
by looking at how elements on $\Lambda(\frakn,\calB)$ overlap. 
For example, triples of form $(1,2,3)$ and $(1,2,4)$ in
  $\Lambda(\frakn,\calB)$
mean that $J_{X_3}$ and $J_{X_4}$ may not be orthogonal, and triples
$(1,3,4)$ and $(2,3,4)$ mean that  $\ad_{X_1}$ and
 $\ad_{X_2}$  may not be orthogonal.  The lack of such ``overlapping
triples'' will imply that the sets are orthogonal.  

Given an inner product space $(\frakn_\mu,Q)$ with orthonormal basis
$\calB$ we can use the corollary to describe membership in
$\calR(\frakn,\calB)$ with polynomial equations in the structure
constants; therefore, $\calR(\frakn,\calB)$ is an algebraic variety in
$\Lambda^2 \frakn^\ast \otimes \frakn.$

Let $\calB = \{X_i\}_{i=1}^n$ be a Ricci eigenvector basis for
 a metric
 algebra $(\frakn_\mu,Q).$ Next we show that the transpose of the
 Ricci vector is a weighted sum of the root vectors for
 $(\frakn_\mu,Q)$ with respect to $\calB,$ where the weight of a root
 vector is $-1/2$ times the square of the corresponding structure
 constant.  This will be illustrated in Example
 \ref{concreteexample1}. The formula was first given as Proposition
 4.1 of \cite{karidi93}, where R.\ Karidi gave a detailed analysis of
 the effects of individual structure constants on the curvature tensor
 of a metric nilpotent Lie algebra.
\begin{theorem}\label{riccitensor} Let 
$(\frakn_\mu,Q)$ be a metric algebra and let $\calB = \{X_i\}_{i=1}^n$
 be a Ricci eigenvector 
basis for $\frakn_\mu.$ Let $\{ Y_{ij}^k \, | \,
 (i,j,k) \in \Lambda(\frakn,\calB)\}$ be the set of root vectors for
 $(\frakn_\mu,Q)$ with respect to $\calB,$ and let $Y$ be the root
 matrix.  Then the eigenvalues $\kappa_{X_1}, \kappa_{X_2}, \ldots,
 \kappa_{X_n}$ of the nil-Ricci endomorphism are given by
\begin{align*}
 \bfRic^\calB_\mu &= (\kappa_{X_1}, \kappa_{X_2}, \ldots,
\kappa_{X_n})^T \\ &= -\half \sum_{(i,j,k) \in \Lambda}
(\alpha_{ij}^k)^2 (Y_{ij}^k)^T = -\frac{1}{2} Y^T [\alpha^2] .
\end{align*}
and $\rho = \trace \Ric_\mu$ is given by
$
\rho = -\half \sum_{(i,j,k) \in \Lambda} (\alpha_{ij}^k)^2 =
-\frac{1}{4} \| \mu \|^2 ,$ where $\mu \in \Lambda^2 \frakn^\ast
\otimes \frakn$ is the Lie bracket for $\frakn.$
\end{theorem}
When $(\frakn_\mu,Q)$ is a metric nilpotent Lie algebra, the number
$\rho$ is the scalar curvature for corresponding nilmanifold.

\begin{proof} 
We use Equation \eqref{nilriccicurvature} to express the eigenvalue
vector for the Ricci endomorphism.
\begin{align*} 
\sum_{k=1}^n \ric_\mu&(X_k,X_k)\varepsilon_k^T \\ &= -\frac{1}{2}
\sum_{i,k=1}^n \| [X_k, X_i] \|^2 \varepsilon_k^T + \frac{1}{4}
\sum_{i,j,k=1}^n \la [X_i,X_j], X_k \ra^2 \varepsilon_k^T \\ &=
-\frac{1}{2} \sum_{i,j,k=1}^n \la [X_k, X_i], X_j \ra^2
\varepsilon_k^T + \frac{1}{4} \sum_{i,j,k=1}^n \la [X_i,X_j], X_k
\ra^2 \varepsilon_k^T \\ \intertext{Breaking the first summand into
two pieces and reindexing the first two pieces, we see this equals} &=
-\frac{1}{4} \sum_{i,j,k=1}^n \la [X_i, X_j], X_k \ra^2
\varepsilon_i^T -\frac{1}{4} \sum_{i,j,k=1}^n \la [X_i, X_j], X_k
\ra^2 \varepsilon_j^T \\ & \qquad + \frac{1}{4} \sum_{i,j,k=1}^n \la
[X_i,X_j], X_k \ra^2 \varepsilon_k^T \\ &= -\frac{1}{4}
\sum_{i,j,k=1}^n (\alpha_{ij}^k)^2 \, Y_{ij}^k \\ &= -\frac{1}{2}
\sum_{1 \le i < j \le n} \sum_{k =1}^n (\alpha_{ij}^k)^2 \, Y_{ij}^k
\intertext{which when summed over nonzero terms becomes} &=
-\frac{1}{2} \sum_{(i,j,k) \in \Lambda} (\alpha_{ij}^k)^2 \, Y_{ij}^k
\end{align*}
Taking the transpose of both sides we get
\begin{equation*}
 \bfRic^\calB_\mu = -\frac{1}{2} \sum_{(i,j,k) \in \Lambda}
(\alpha_{ij}^k)^2 (Y_{ij}^k)^T, \end{equation*} completing the proof.

 \end{proof}

\subsection{Sectional curvature}
A Lie algebra $\frakg$ is said to be ${\boldR}^+$-{\em graded} if it
can be written as $\frakg = \frakg_{\lambda_1} \oplus \cdots \oplus
\frakg_{\lambda_r}$ where $[\frakg_{\lambda_i},\frakg_{\lambda_j}]
\subset \frakg_{\lambda_{i} + \lambda_j}$ ($\frakg_{\lambda_k} := 0$
if $\lambda_k > \lambda_r$) and $\lambda_1, \lambda_2, \ldots,
\lambda_r$ are in $\boldR^+.$ An $\boldR^+$-graded algebra is
nilpotent.  For a nilmanifold, tangent two-planes spanned by
two vectors in  subspace summands of
 a ${\boldR}^{+}$-grading, the sectional curvature has
a particularly nice expression.
\begin{theorem}\label{Ricci-K}
  Let $(\frakn,Q)$ be a metric Lie algebra.  Suppose that $\frakn$ has
a orthogonal ${\boldR}^{+}$-grading $\frakn = \oplus
\frakn_{\lambda}.$ For unit $X$ in $\frakn_{\mu}$ and unit $Y$ in
$\frakn_{\nu}$ orthogonal to $X,$ if $\mu \le \nu,$ the sectional
curvature of the plane spanned by $X$ and $Y$ is given by
\[  K(X \wedge Y) = -\smallfrac{3}{4} \|  \ad_X Y \|^2 
+\smallfrac{1}{4} \|J_Y X \| ^2 . \] 
If in addition $X= X_i,$ $Y=Y_j$ for
an orthonormal basis $\{X_i\}_{i=1}^n,$ then
\[ K(X_i \wedge X_j) = 
-\frac{3}{4} \sum_{(i,j,k) \in \Lambda(\frakn,\calB)}
(\alpha_{ij}^k)^2 + \frac{1}{4} \sum_{(i,k,j) \in
  \Lambda(\frakn,\calB)} (\alpha_{ij}^k)^2. \]
\end{theorem}

  The positive definite derivation $\Ric - \beta\Id$ for a metric
nilpotent algebra satisfying the nilsoliton condition with nilsoliton
constant $\beta < 0$ defines an $\boldR^+$ grading, so we can use the
previous theorem to find sectional curvatures for nilsolitons.

\begin{proof} 
Suppose that $X$ and $Y$ are orthogonal unit vectors in $\frakg_{\mu}$
 and $\frakg_{\nu}$ respectively, and $\mu \le \nu.$ Since $\oplus
 \frakg_\lambda$ is a grading, $\ad_X Y$ is in $\frakg_{\mu + \nu}.$
 We claim that for any $W$ in any $\frakg_\lambda,$ the vector $J_{Y}X$
 is in $\frakg_{\nu - \mu}.$ The inner product $ \la Y, \ad_X W \ra
 = \la J_Y X, W\ra $ is nonzero only if $\nu = \mu + \lambda,$ by the
 orthogonality of the grading. This shows that $J_{Y} X$ is in
 $\frakg_{\nu - \mu}.$ The same argument shows that $J_X Y = 0$ since
 $\nu - \mu \le 0$ makes $\frakg_{\mu - \nu}$ trivial.

Now we use Equation \eqref{connection} to compute
\begin{align*}
\| \nabla_ X Y \|^2 &= \smallfrac{1}{4} \la \ad_X Y - J_Y X , \ad_X Y
- J_Y X \ra \\ &= \smallfrac{1}{4} \left( \|\ad_X Y \|^2 - 2 \la \ad_X
Y, J_Y X \ra + \| J_Y X\|^2\right) \\ &= \smallfrac{1}{4} \left(
\|\ad_X Y \|^2 + \| J_Y X\|^2\right)
\end{align*}
The term $\la \ad_X Y, J_Y X \ra$ vanishes since $\ad_X Y$ is in
$\frakg_{\mu + \nu}$ and $ J_Y $ is in $\frakg_{\nu - \mu},$ and these
orthogonal spaces can not coincide unless $\mu = 0.$ Equation
\eqref{connection} also shows that $\nabla_W W = 0$ for any $W$ in any
$\frakg_\lambda$ because $\ad_X^\ast X = J_X X$ is in $\frakg_0 = \{ 0 \}.$

Substituting the expressions for $\nabla_X X, \nabla_Y Y$ and $\|
\nabla_X Y\|^2$ into Equation \eqref{K} and simplifying, we get
\begin{align*} K(X \wedge Y) &= \smallfrac{1}{4} ( \|\ad_X Y \|^2  + 
\| J_Y X\|^2) - \|\ad_X Y \|^2 \\ &= -\smallfrac{3}{4} \|\ad_X Y \|^2 +
\smallfrac{1}{4}\| J_Y X\|^2 \end{align*} as desired.  The equation
for $K(X_i \wedge X_j)$ follows from writing the equation for $K(X
\wedge Y)$ relative to an orthonormal basis and setting terms not in
$\Lambda(\frakn_\mu, Q)$ equal to zero.   
 \end{proof}

\subsection{Examples}
We illustrate the preceding theorems with an example.
\begin{example}\label{concreteexample1}
Let $(\frakn_\mu, Q)$ be the metric nilpotent Lie algebra with
orthonormal basis $\calB = \{X_1, X_2, X_3, X_4, X_5 \}$ and bracket
relations
\[ [X_1,X_2] = a X_3,  [X_1,X_3] = bX_4, [X_1,X_4]=cX_5,
 [X_2,X_3] = dX_5 ,\] where $\mu$ is a function of $a,b,c,d \in
\boldR \setminus \{ 0 \}.$ Here we have
\[\Lambda(\frakn,\calB) = \{ (1,2,3), (1,3,4), (1,4,5), (2,3,5)\}.\] 
Notice that from the way triples in
$\Lambda(\frakn,\calB)$ overlap,  the sets
$\{\ad_{X_i}\}_{i=1}^5$ and $\{J_{X_i}\}_{i=1}^5$ are both orthogonal
in $\End(\frakn)$ 
 for all $a,b,c,d,$ and then by Corollary \ref{orthogonal},
$\{X_i\}_{i=1}^5$ is a Ricci eigenvector basis.  The root matrix for
$(\frakn_\mu, Q)$ with respect to $\calB$ is $\D{Y = \begin{bmatrix} 1
& 1 & -1 & 0 & 0 \\ 1 & 0 & 1 & -1 & 0 \\ 1 & 0 & 0 & 1 & -1 \\ 0 & 1
& 1 & 0 & -1 \\
\end{bmatrix}}$ and the structure vector is 
$[\alpha^2] = (a^2, b^2, c^2, d^2)^T.$ By Theorem \ref{riccitensor},
the Ricci vector $\bfRic^\calB_\mu$ is given by
\begin{multline*}
  -\half ( a^2 \, Y_{12}^3+ b^2 \, Y_{13}^4 + c^2 \, Y_{14}^5+ d^2 \,
Y_{23}^5 )^T = \\ -\half
(a^2+b^2+c^2,a^2+d^2,-a^2+b^2+d^2,-b^2+c^2,-c^2-d^2)^T,
\end{multline*}
and hence the eigenvalues for the Ricci endomorphism are
\[  -\half (a^2+b^2+c^2),\half (a^2 + d^2),
\half (a^2-b^2-d^2), \half (b^2-c^2),\half (c^2 + d^2) \] Thus, with
respect to the basis $\calB,$ the Ricci endomorphism is represented by
the diagonal matrix $\diag(\bfRic^\calB_\mu).$

Notice that $\frakn$ has a grading defined by $ \frakn = \oplus_{i =
1}^5 \frakn_i$ with $\frakn_i = \la X_i \ra$ for $i = 1, \ldots, 5.$
As an illustration of Theorem \ref{Ricci-K}, we note that the
sectional curvature of the plane spanned by $X_1$ and $X_4$ is
$\frac{1}{4}(b^2 - 3c^2).$

The Gram matrix for $\Lambda(\frakn,\calB)$ is
\[ U = YY^T = \begin{bmatrix} 3 & 0 & 1 & 0 \\  0 & 3 & 0 & 1
\\ 1 & 0 & 3 & 1 \\ 0 & 1 & 1 & 3 \\ \end{bmatrix},\] and an
associated generalized Cartan matrix is
\[ A= U - [1]_{4 \times 4} = 
\begin{bmatrix} 2 & -1 & 0 & -1 \\  -1 & 2 & -1 & 0
\\ 0 & -1 & 2 & 0 \\ -1 & 0 & 0 & 2 \\ \end{bmatrix}.\] It may be
checked that the graphs $S(U)$ and $S(A)$ determined by $U$ and $A$
respectively both have the same incidences as type $A_4.$ Note that
because zero and nonzero off-diagonal entries are reversed in $U$ and
$A$ that one graph can be easily obtained from the other.
\end{example}

In \cite{lauret02}, using a variational approach, J.\ Lauret gave the
first presentation of nonsymmetric Einstein solvmanifolds of
arbitrarily high step, as extensions of nilsoliton Lie algebras of
arbitrarily high step.  In the next example, we use Theorem
\ref{riccitensor} to compute the Ricci curvature in this family.  We
will revisit this example later.

\begin{example}\label{Ln example, part one}
Let $\frakn_\mu$ be an $n$-dimensional metric nilpotent Lie algebra
 with a codimension one abelian ideal $\frakm.$ Let $X$ be a vector
 not in $\frakm.$ The mapping $\ad_X |_\frakm$ is nilpotent, so there
 exists a Jordan decomposition of $\frakm$ as the direct sum of $r$
 subspaces $\frakm_1, \ldots, \frakm_r$ of dimensions $n_1 + 1, n_2 +
 1, \ldots, n_r + 1$ respectively, such that for each $i = 1, \ldots,
 r,$ there is a basis $\{X_{ij} \, | \, 1 \le j \le n_i + 1\}$ for
 $\frakm_i$ so that $\ad_X$ is given by $[X, X_{ij}] = a_{ij} X_{i(j +
 1)}$ ($X_{i(n_1 + 2)} := 0$).

In terms of the lexicographically ordered basis $\{ X_{ij} \},$ we may
write $\ad_X|_\frakm$ as
\[ \ad_X|_\frakm = \begin{bmatrix} J_{n_1} & & \\  & \ddots & \\ 
& & J_{n_r}\end{bmatrix}, \qquad \text{\rm where} \qquad J_{n_i} =
\begin{bmatrix} 0 & & & \\ a_{i1} & \ddots & & \\ & \ddots & \ddots &
\\ & & a_{in_i} & 0 \end{bmatrix} . \] We impose the inner product $Q$
on $\frakn_\mu$ that has $\calB = \{X\} \cup \{X_{ij} | \, 1 \le i \le
r, 1 \le j \le n_i + 1\}$ as an orthonormal basis.  The basis $\calB$
is a Ricci eigenvector basis by Corollary \ref{orthogonal}, because
$\{ \ad_{X_{ij}}\}$ and $\{J_{X_{ij}}\}$ are both orthogonal sets.

Observe that $J_X \equiv 0,$ so by Theorem \ref{niceformula},
\[ \ric_\mu(X,X) =  -\frac{1}{2} \|\ad_X \|^2 
= -\frac{1}{2} \sum_{1 \le i \le r} \, \sum_{1 \le j \le n_i + 1 }
a_{ij}^2. \] Next we use Theorem \ref{riccitensor} to find the
restriction of the Ricci endomorphism to $\frakm_i$ for each $i.$
Using that $X_{i(n_i+2)}:=0,$ we get that the projection of the Ricci
vector $\bfRic^\calB_\mu$ to $\frakm_i$ is
\[ \frac{1}{2}(-a_{i1}^2, a_{i1}^2-a_{i2}^2, \ldots,  a_{i(n_i - 1)}^2 -
a_{in_i}^2, a_{in_i}^2),\] and hence
\begin{align*} 
\ric_\mu(X_{i1},X_{i1}) &= -\half a_{i1}^2, \\ \ric_\mu(X_{ij},X_{ij})
&= \half (a_{i(j-1)}^2 - a_{ij}^2) \text{\rm \quad for $j = 2$ to
$n_i$} \\ \ric_\mu(X_{in_{i+1}},X_{in_{i+1}}) &= \half a_{in_i}^2 .
\end{align*}

The Gram matrix $U$ and generalized Dynkin diagram $S(U)$ for
 $(\frakn_\mu,Q)$ relative to $\calB$ are not particularly nice;
 however the associated matrix $A = U - \onematrix{n}{n}$ is of type
 $A_{n_1 - 1} + \cdots + A_{n_r - 1}.$
\end{example}

\section{Properties of the Gram matrix $U$ and associated matrix $A$}
\label{properties}

In the next proposition, we summarize important properties of the Gram
matrix.
\begin{theorem}\label{gramproperties} Let $(\frakn, Q)$ be an $n$-dimensional 
inner product space with Ricci eigenvector basis $\calB.$ Let $\mu$ be
a $\Lambda^2 \frakn^\ast \otimes \frakn$ so that
$\Lambda(\frakn_\mu,\calB) \subset \Psi_n.$ Let $U=(u_{ij})$ be the
Gram matrix for $(\frakn_\mu, Q)$ with respect to $\calB.$ Then
\begin{enumerate}
\item{The diagonal entries of $U$ are all three}
\item{The off-diagonal entries of $U$ are in the set $\{-2,-1, 0, 1, 2
\}.$}
\item{$U$ is symmetric}
\item{The rank of $U$ is equal to the number of independent vectors in
the set $\{Y_{ij}^k \, | \, (i,j,k) \in \Lambda(\frakn,\calB) \}$ of
root vectors for $(\frakn_\mu,Q)$ with respect to $\calB.$ }
\item{$U$ is positive semidefinite.}
\end{enumerate}
Furthermore, the associated matrix $A$ defined by $A = 2U -
4\onematrix{m}{m},$ where $m = \| \Lambda\|,$ is a generalized Cartan
matrix.  If $U$ has no entries of $2,$ then $A^\prime = U -
\onematrix{m}{m}$ is a generalized Cartan matrix.  Also, $U$ is
positive definite if either $A$ or $A^\prime$ is positive definite
\end{theorem}

\begin{proof}
By skew-symmetry and the hypothesis that $\Lambda \subset \Psi_n$,
 $\alpha_{ij}^k = 0$ if $i, j,$ and $k$ are not all distinct, so if
 $(i,j,k) \in \Lambda,$ the structure vector $Y_{ij}^k$ is represented
 by a vector with two entries of one and one entry of $-1,$ with
 zeroes comprising the remaining entries.  The $(i,j)$ entry of $U$ is
 the inner product of the $i$th structure vector $Y_i$ and the $j$th
 structure vector $Y_j.$ After considering all the possibilities for
 the arrangements of the nonzero entries of $Y_i$ and $Y_j,$ one sees
 that the inner product $\la Y_i, Y_j \ra$ must be in the set $\{ -2,
 -1, 0, 1, 2, 3\},$ with value three if and only if $Y_i = Y_j.$ As
 the root vectors are all distinct, this occurs if and only if $i=j.$
 Thus, properties (1) and (2) hold.  The definition of $U$ as $YY^T$
 insures that properties (3), (4) and (5) hold (\cite{hornjohnson}, Theorem
 7.2.10).

Properties (1)-(3) of generalized Cartan matrices hold for $A$ from the
definition of associated matrix, the symmetry of $U,$ and the fact
that $\onematrix{m}{m}$ is a symmetric matrix.  For any $n \times 1$
vector $v,$ $v A v^T = 2 v U v^T - 4 v v^T$ and $v A^\prime v^T = v U
v^T - vv^T$ so if $v A v^T > 0$ or $v A^\prime v^T > 0,$ then $v U v^T
> 0$.    \end{proof}

\section{The Jacobi condition}\label{jacobi}

The next theorem gives a simple way to check if a set of nonzero
structure constants indexed by a subset of the set $\Theta_n$ define a
Lie algebra.  We will use the theorem later to check whether a metric
algebra $\frakn_\mu$ satisfies the Jacobi condition when it is
represented with respect to a triangular orthonormal basis.
\begin{theorem}\label{jacobicondition} Let $(\frakn,Q)$ 
be an $n$-dimensional inner product space, and let $\mu$ be an element
of $\Lambda^2 \frakn^\ast \otimes \frakn.$ Let $\calB =
\{X_i\}_{i=1}^n$ be an orthonormal basis for $(\frakn,Q)$ and suppose
that $\mu$ is $\calB$-triangular.  Then
\begin{enumerate}
\item{If $i,j$ and $k$ are distinct, the product $\alpha_{ij}^l
 \alpha_{lk}^m \ne 0$ if and only if $\la Y_{ij}^l, Y_{lk}^m \ra =
 -1.$}
\item{The algebra $\frakn_\mu$ defined by $\mu$ is a Lie algebra if
and only if whenever there exists $m$ so that $ \la Y_{ij}^l, Y_{lk}^m
\ra = -1$ for triples $(i,j,l)$ and $(l,k,m)$ or $(k,l,m)$ in
$\Lambda(\frakn_\mu,\calB),$ the equation
\begin{equation}\label{jacobi-constraint} 
\sum_{s < m} \alpha_{ij}^s \alpha_{sk}^m + \alpha_{jk}^s \alpha_{si}^m
+ \alpha_{ki}^s \alpha_{sj}^m = 0
\end{equation} holds.}
\end{enumerate}
\end{theorem}
Recall that the inner product of two root vectors is minus one if and
only if the corresponding entry of the Gram matrix is minus one.  One
consequence of this theorem is that the generalized Dynkin diagram
$S(U)$ for a nilsoliton metric Lie algebra can not have exactly one
dashed edge with weight minus one, or an equation like
Equation \eqref{jacobi-constraint} would need to be satisfied, but
would not be able to hold. 
 Another consequence is that if
there are no minus one edges, the Jacobi condition holds
automatically:
\begin{corollary}\label{jacobicondition-cor} 
Let $(\frakn,Q)$ be an $n$-dimensional inner product space, and let
$\mu$ be an element of $\Lambda^2 \frakn^\ast \otimes \frakn.$ Let
$\calB = \{X_i\}_{i=1}^n$ be an orthonormal basis for $(\frakn,Q)$ and
suppose that $\mu$ is $\calB$-triangular.  If the Gram matrix $U$ for
$(\frakn_\mu,Q)$ has no entries of -1, then $(\frakn_\mu,Q)$ is a Lie
algebra.
\end{corollary}
Equation \eqref{jacobi-constraint} may have continuous families of
solutions, or more than one finite solution (because of sign choices
in solving equations with squares).  It is possible for different sign
choices for the structure constants to give different nilpotent
structures on an inner product space $\frakn,$ and hence nonisometric
nilsolitons; i.e.\ there exists $\Lambda$ such that there are
nonisometric $\mu$ in $\Omega_\Lambda$ with the same structure vector.
This occurs in dimension six (see $\mu_7$ and $\mu_8$ in
\cite{will03}).  Two such nilpotent structures have the same Ricci
endomorphisms and derivations $D = \Ric - \beta \Id$, so there is not
a one-to-one correspondence between nilsoliton metrics and eigenvalues
of the derivation $D.$
Now we prove Theorem \ref{jacobicondition}.

\begin{proof}
The Jacobi identity holds for basis vectors $X_i, X_j,$ and $X_k$ if
and only if
\begin{equation}\label{jacobiidentity} 
\sum_{s = 1}^n \alpha_{ij}^s \alpha_{sk}^m + \alpha_{jk}^s
\alpha_{si}^m + \alpha_{ki}^s \alpha_{sj}^m = 0 \end{equation} for all
$m.$ This equation holds for $X_i, X_j$ and $X_k$ by skew-symmetry if
any index $i, j$ or $k$ is repeated, so the Jacobi condition holds if
and only if Equation \eqref{jacobiidentity} is true for all $i,j,k$ and
$m$ where $i, j, $ and $k$ are distinct.

Now we will establish that when $i, j$ and $k$ are distinct, a product of
the form $\alpha_{ij}^l \alpha_{lk}^m$ is nonzero if and only if $\la
Y_{ij}^l, Y_{lk}^m \ra = -1$ for root vectors $Y_{ij}^l$ and
$Y_{lk}^m.$ Suppose that $\alpha_{ij}^l \alpha_{lk}^m$ is nonzero. As
$\alpha_{ij}^l \ne 0$ and $\alpha_{lk}^m \ne 0$ and $\mu$ is
$\calB$-triangular by hypothesis, we have $i, j < l$ and $l, k < m.$
It follows that $i < m$ and $j < m.$ We also know that $l \ne k$ since
$\alpha_{lk}^m$ is nonzero, and $i \ne k$ and $j \ne k$ because $i,
j,$ and $k$ were assumed to be distinct.  The inner product of $
Y_{ij}^l$ and $Y_{lk}^m$ is then
\begin{align*} \la  Y_{ij}^l, Y_{lk}^m \ra &= 
\la \varepsilon_i^T + \varepsilon_j^T - \varepsilon_k^T,
\varepsilon_l^T + \varepsilon_k^T - \varepsilon_m^T \ra \\ &=
\delta_{il} +\delta_{ik} - \delta_{im} +\delta_{jl} +\delta_{jk} -
\delta_{jm} - \delta_{ll} - \delta_{lk} + \delta_{lm} \\ &= -1
. \end{align*} For the converse, suppose that the inner product of two
root vectors is minus one.  Then they must be of the form $Y_{ij}^l$
and $Y_{lk}^m,$ or $Y_{ij}^l$ and $Y_{kl}^m.$ In the first case,
$(i,j,l)$ and $(l,k,m)$ are in $\Lambda(\frakn_\mu,Q)$ so
$\alpha_{ij}^l \alpha_{lk}^m \ne 0 ,$ and in the second case,
$(i,j,l)$ and $(k,l,m)$ are in $\Lambda(\frakn_\mu,Q)$ so
$\alpha_{ij}^l \alpha_{lk}^m = -\alpha_{ij}^l \alpha_{kl}^m \ne 0 .$
We have proved the asserted equivalence in the first part of the
theorem.

Now we suppose that $(\frakn_\mu,Q)$ satisfies the Jacobi identity, so
Equation \eqref{jacobiidentity} holds for all $i,j,k$ and $m.$ Fix $m$
and suppose that the equation holds trivially.  Then there are no
nonzero terms of the form $\alpha_{ij}^l \alpha_{kl}^m$ so there are
no root vectors  whose inner product is minus one.
If Equation \eqref{jacobiidentity} does not hold trivially for some
value of $m,$ we can simplify it slightly by noting that terms of the
form $\alpha_{ij}^s \alpha_{ks}^m$ vanish when $s \ge m$ because $\mu$
is $\calB$-triangular.  This establishes the second assertion of the
theorem.    \end{proof}

\section{Conditions equivalent to the nilsoliton condition}
\label{proof of main thm}

\subsection{Derivations, gradings, and the nilsoliton condition}
First we prove some theorems relating derivations, gradings, the
nilsoliton condition, and eigenvalues.  In the proof of Theorem 4.14
of \cite{heberinv}, Heber characterized derivations of a metric
nilpotent algebra $(\frakn,Q)$ in terms of vectors of the form
$\varepsilon_i^T + \varepsilon_j^T - \varepsilon_k^T$ like the root
vectors used in this paper.  The next theorem is restatement and
slight generalization of Heber's idea.
For $(\mu_1, \ldots, \mu_n)^T$ in $\boldR^n,$ define
\begin{align*}
 \Sigma_{(\mu_1,\ldots,\mu_n)} &= \{(i,j,k) \in \Upsilon_n \, | \,
\mu_i + \mu_j = \mu_k \} \\ &= \{ (i,j,k) \in \Upsilon_n \, | \,
Y_{ij}^k \, (\mu_1, \ldots, \mu_n)^T = 0 .\}
\end{align*}
We say that a generalized Dynkin diagram $S(U)$ is a {\em minor graph}
of another generalized Dynkin diagram $S(U^\prime)$ if the vertices of
the first are a subset of the vertices of the second, and the
generalized Cartan matrix for the first is the corresponding minor of
the second.

\begin{theorem}\label{hebervar}  
Let $(\frakn,Q)$ be a $n$-dimensional nonabelian metric algebra with
orthonormal basis $\calB = \{X_i\}_{i = 1}^n.$ Let $D$ be the
endomorphism of \frakn defined by $D(X_i) = \mu_i X_i$ for $1 =
1,\ldots, n,$ and let $v_D = (\mu_1, \ldots, \mu_n)^T.$

 The following properties are equivalent:
\begin{enumerate}
\item{$D$ is a derivation,}
\item{ $\Lambda(\frakn,\calB)$ is a subset of $\Sigma_{(\mu_1,\ldots,
 \mu_n)},$ }
\item{$v_D$ is in the kernel of the root matrix
$Y_{\Lambda(\frakn,\calB)}$ for
$\Lambda(\frakn,\calB).$}\label{kernelproperty}
\item{$S(\Lambda(\frakn,\calB))$ is a minor graph of
$S(\Sigma_{(\mu_1,\ldots, \mu_n)})$}
\end{enumerate}
Furthermore, if $\Lambda(\frakn,\calB)$ is a subset of a set of
integer triples $\Gamma \subset \Upsilon_n,$ the kernel of the root
matrix for $\Gamma$ is a subset of kernel of the root matrix for
$\Lambda(\frakn,\calB).$
\end{theorem}

The theorem allows us to interpret the set  $\Sigma_{(\mu_1,\ldots,
 \mu_n)}$ as the admissible  nonzero structure constants  for
algebras admitting a derivation with eigenvalues $\mu_1,\ldots,
 \mu_n.$ 
\begin{proof}  
Applying the definition of derivation to pairs of basis vectors, we
see that $D$ is a derivation if and only if $[X_i,X_j]$ lies in the
$\mu_i + \mu_j$ eigenspace for all basis vectors $X_i$
and $X_j.$ Equivalently, since
\[ \la D([X_i,X_j]), X_k \ra  = \la 
\sum_{l = 1}^n \alpha_{ij}^l D(X_l), X_k \ra = \alpha_{ij}^k \mu_k, \]
if $\alpha_{ij}^k \ne 0,$ then $\mu_k = \mu_i + \mu_j.$ Then $D$ is a
derivation if and only if whenever the triple $(i,j,k)$ is in
$\Lambda(\frakn,\calB),$
\[ 0 =  \mu_i + \mu_j - \mu_k = 
 Y_{ij}^k (\mu_1, \ldots, \mu_n)^T = Y_{ij}^k v_D ;\] or
$\Lambda(\frakn,\calB) \subset \Sigma_{(\mu_1, \ldots, \mu_n)}.$ This
confirms the equivalence of the 
first and second properties.

Recall that because the rows of the root matrix $Y$ for a set of
 triples $\Gamma$ are the vectors $\{ Y_{ij}^k \, | \, (i,j,k) \in
 \Gamma \},$ a vector $v_D$ is in the kernel of the root matrix for
 $\Gamma$ if and only if $Y_{ij}^k v_D = 0$ for all $ (i,j,k) \in
 \Gamma.$ By this fact, the third property holds if and only if
 $Y_{ij}^k v_D = 0$ for all $(i,j,k)$ in $\Lambda(\frakn,\calB),$
 which occurs if only if $\Lambda(\frakn,\calB)$ is a subset of
 $\Sigma_{(\mu_1, \ldots, \mu_n)}.$ Therefore the second and third
 properties are equivalent.  Clearly the second and fourth properties
 are equivalent.

It is immediate from the same fact that whenever $\Lambda \subset
\Gamma \subset \Upsilon_n,$ the rows of the root matrix for $ \Lambda$
are rows of the root matrix for $\Gamma.$ If a vector $v$ is in the
kernel of the root matrix for $\Gamma,$ then $v$ is in the kernel of
the root matrix for $\Lambda.$ Letting $\Lambda =
\Lambda(\frakn,\calB)$ concludes the proof.
\end{proof}

It was shown in \cite{heberinv} that for a a metric Lie algebra
$(\frakn,Q)$ there is a unique $\beta$ so that $\Ric - \beta \Id$ is a
derivation, and it is
\[\beta = - \smallfrac{\trace D^2}{\trace D} = 
-\smallfrac{\sum \mu_i^2}{\sum \mu_i}.\] If $v_D = (\mu_1, \ldots,
\mu_n)$ is such a derivation, then we call $(\mu_1, \ldots, \mu_n)$
the {\em eigenvalue type} of $(\frakn,Q).$ The following essential
theorem translates the nilsoliton condition into a linear condition on
the eigenvalues of the Ricci endomorphism.  Theorem 1.6 of
\cite{wolter91a} is an analogous theorem for Einstein solvmanifolds.
\begin{theorem}\label{nilsolitonconstraint}  
Let $(\frakn_\mu, Q)$ be a nonabelian metric algebra with Ricci
 eigenvector basis $\calB.$ Let $D = \Ric - \beta \Id.$ The following
 are equivalent:
\begin{enumerate}
\item{$(\frakn_\mu, Q)$ satisfies the nilsoliton condition with
nilsoliton constant $\beta$}
\item{The eigenvalue vector $ v_D$ for $D$ with respect to $\calB$
lies in the kernel of the root matrix for $(\frakn,\calB)$ with
respect to $\calB.$}
\item{For noncommuting eigenvectors $X$ and $Y$ for the nil-Ricci
endomorphism with eigenvalues $\kappa_X$ and $\kappa_Y$, the bracket
$[X,Y]$ is an eigenvector for the nil-Ricci endomorphism with
eigenvalue $\kappa_X + \kappa_Y - \beta.$ }
\item{$Y_{ij}^k \, \bfRic = \beta $ for all $(i,j,k)$ in
$\Lambda(\frakn_\mu,\calB).$}
\end{enumerate}
When $(\frakn_\mu, Q)$ satisfies the soliton condition with nilsoliton
constant $\beta$, the eigenspaces for the derivation $\Ric_\mu - \beta
\Id$ are orthogonal.
\end{theorem}

\begin{proof} 
As $\Ric$ and $D$ differ by a multiple of the identity, the basis
$\calB$ simultaneously diagonalizes the two maps $\Ric$ and $D,$ with
the $\kappa$ eigenspace for $\Ric$ equal to the $\kappa - \beta$
eigenspace for $D.$ By definition, the metric Lie algebra
$(\frakn_\mu, Q)$ satisfies the nilsoliton condition with nilsoliton
constant $\beta$ if and only if $D = \Ric - \beta \Id$ is a
derivation.  The map $D$ is a derivation if and only if for
noncommuting eigenvectors $X$ and $Y$ of $D$ with eigenvalues $\mu_X$
and $\mu_Y$, the bracket $[X,Y]$ is an eigenvector for the nil-Ricci
endomorphism with eigenvalue $\mu_X + \mu_Y.$ The $\mu$ eigenspace of
$D$ is equal to the $\mu + \beta$ eigenspace of $\Ric;$ translating
the equivalence into a statement about $\Ric$ eigenspaces shows the
equivalence of the first and third properties.
 
By Theorem \ref{hebervar}, $D$ is a derivation if and only if $v_D$ is
in the kernel of $Y.$ Hence the first two properties are equivalent.

Assume $(i,j,k)$ is in $\Lambda(\frakn_\mu, \calB).$ By Theorem
\ref{riccitensor}, $Y_{ij}^k \bfRic = \kappa_{X_i} + \kappa_{X_j}-
\kappa_{X_k},$ hence $Y_{ij}^k \bfRic = \beta$ if and only if
$\kappa_{X_k} = \kappa_{X_i} + \kappa_{X_j}- \beta.$ The equivalence
of the third and fourth properties follows from writing arbitrary
eigenvectors $X$ and $Y$ in terms of the basis eigenvectors.

 The subspaces of the grading for $(\frakn_\mu,Q)$ are orthogonal
because they are eigenspaces of the nil-Ricci endomorphism.   
\end{proof}

J.\ Heber has shown that after rescaling, the eigenvalues for $D =
\Ric - \beta \Id$ for a nilsoliton $(\frakn,Q)$ are positive integers
with no common divisors (Theorem C, \cite{heberinv}).  The previous
theorem can be used in an elementary argument for why the nilsoliton
constant $\beta$ for a nonabelian nilmanifold must be negative.  Were it not,
then by the previous theorem, 
the span of all Ricci eigenvectors with negative eigenvalues
would be a subalgebra.  But the Ricci endomorphism is negative
semidefinite when restricted to the generating subspace
$[\frakn,\frakn]^\perp$ by Theorem \ref{niceformula}, so then $\Ric$
would be negative definite on all of $\frakn.$ This is a
contradiction, since again by Theorem \ref{niceformula}, the Ricci
endomorphism is positive definite on the center.

The next lemma is elementary.
\begin{lemma} Let $(\frakn_\mu,Q)$ 
be a metric algebra with orthonormal basis $\calB = \{X_i\}_{i =
1}^n.$ Suppose $\frakn_\mu$ has an $\boldR^+$-grading $\frakn_{\mu_i}
= \oplus_{\mu_i \in \boldR^+} \frakn_\mu,$ such that each basis vector
$X_i$ is in some subspace $\frakn_{\mu_i}$ for the grading.  Then
$\Lambda(\frakn_\mu,\calB) \subset \Sigma_{(\mu_1, \ldots, \mu_n)}
\subset \Theta_n$.
\end{lemma}

\begin{proof}
Since $\oplus_{\mu \in \boldR^+} \frakn_\mu$ is a grading,
 $[\frakn_{\mu},\frakn_{\nu}]$ lies in $\frakn_{\mu + \nu}$ for all
 $\mu$ and $\nu$ in $\boldR^+.$ This directly implies that
 $\Lambda(\frakn_\mu,\calB) \subset \Sigma_{(\mu_1, \ldots, \mu_n)}$.

The space $\frakn_{\mu + \nu}$ intersects both of the spaces
 $\frakn_\mu$ and $\frakn_\nu$ trivially as $\mu + \nu$ is greater
 than both $\mu$ and $\nu.$ The orthogonality of the eigenspaces and
 the compatibility of the basis with the grading imply that $\la
 [X_i,X_j], X_i \ra = 0$ and $\la [X_i,X_j], X_j \ra = 0,$ so
 structure constants of the form $\alpha_{ij}^i$ are always zero, as
 claimed.    \end{proof}

The positivity of the eigenvalues for $D$ also gives the following
useful corollary to the previous lemma.
\begin{corollary}\label{goodbasis} 
Let $(\frakn_\mu,Q)$ be a nilsoliton metric algebra with Ricci
eigenvector basis $\calB.$ Then $\mu$ is $\calB$-triangular.
\end{corollary}

\begin{proof}
By definition, the Ricci eigenvector basis is orthonormal.  Suppose
that the nilsoliton constant for $(\frakn,Q)$ is $\beta$ and let $D =
\Ric - \beta \Id.$ Then $\calB$ is also an orthonormal basis of
eigenvectors for $D.$ The eigenvalues of $D$ are all positive, so $D$
defines an $\boldR^+$ grading of $\frakn.$   \end{proof}

\subsection{Proof of 
Theorem \ref{mainthm}.}

\begin{proof}
Let $D = \Ric - \beta\Id.$ The Ricci vector $\bfRic$ is the eigenvalue
vector for the nil-Ricci endomorphism with respect to a Ricci
eigenvector basis.  Let $v_D$ be the eigenvalue vector for $D$ with
respect to $\calB.$ It is given by
\[ v_D = v_{\Ric} - \beta v_{\Id} = \bfRic  - \beta \onevector{n}. \]
By Theorem \ref{riccitensor}, the Ricci vector equals $ -\frac{1}{2}
 Y^T [\alpha^2], $ where $[\alpha^2]$ is the structure vector for
 $(\frakn,Q)$ with respect to $\calB.$ Hence
\[  v_D =   -\smallfrac{1}{2}   Y^T [\alpha^2] - \beta \onevector{n}.\]
Multiplying both sides of the equation by the root matrix yields
\[ Y  v_D =   -\smallfrac{1}{2}  Y  Y^T [\alpha^2] - \beta Y \onevector{n}\] 
For all triples $(i,j,k),$ we have $Y_{ij}^k \onevector{n} =
\onevector{1},$ so $Y \onevector{n} = \onevector{n}.$ Making this
substitution and noting that $Y Y^T = U,$ we find that
\[ Y  v_D =   -\smallfrac{1}{2} U [\alpha^2] - \beta  \onevector{n}.\]
Therefore, $Y v_D = 0$ if and only if $U [\alpha^2] = -2\beta
\onevector{n}.$ By Theorem \ref{hebervar}, $Y v_D = 0$ if and only if
$D$ is a derivation.  Thus, as asserted, $(\frakn,Q)$ satisfies the
nilsoliton condition with nilsoliton constant $\beta$ if and only if
$U [\alpha^2] = -2\beta\onevector{n}.$
\end{proof}

\subsection{Examples}
We next illustrate Theorem \ref{mainthm} and Theorem
 \ref{jacobicondition} with a continuation of Example
 \ref{concreteexample1}.
\begin{example}\label{concreteexample1-b}  Let $(\frakn,Q)$ be the metric
nilpotent algebra with Ricci eigenvector basis $\calB$ from Example
\ref{concreteexample1}. The set of root vectors $\{ Y_{12}^3 ,
Y_{13}^4 , Y_{14}^5 , Y_{23}^5 \}$ is linearly independent, so the
matrix $U$ is invertible.  Thus, the system $Uv=-2\beta \onevector{4}$
is consistent. The solution is
\[ \begin{bmatrix}
a^2 \\ b^2 \\ c^2 \\ d^2
\end{bmatrix} = -2 \beta \begin{bmatrix} 3 & 0 & 1 & 0 \\ 
0 & 3 & 0 & 1 \\ 1 & 0 & 3 & 1 \\ 0 & 1 & 1 & 3 \\
\end{bmatrix}^{-1}
 \begin{bmatrix}  
1 \\ 1 \\ 1 \\ 1
\end{bmatrix} =  -\frac{2\beta}{11}
\begin{bmatrix} 3 \\ 3 \\ 2\\ 2  \end{bmatrix}. \]

We normalize $\beta$ to $-11/2$ and take positive roots, so
$a=b=\sqrt{3}$ and $c=d=\sqrt{2}$ are structure constants defining a
metric algebra satisfying the nilsoliton condition by Theorem
\ref{mainthm}.  There are no entries of minus one in the Gram matrix,
so by Corollary \ref{jacobicondition}, $\frakn_\mu$ satisfies the
Jacobi identity and is a Lie algebra.  Up to rescaling and metric
isomorphism, this is the unique nilsoliton metric in the equivalence
class $\Omega_{\Lambda} \subset \calR(\frakn,\calB).$ The derivation
$D = \Ric_\mu - \beta \Id$ has eigenvalue vector
\[ (\mu_1,\mu_2,\mu_3,\mu_4,\mu_5)^T 
= \bfRic_\mu^\calB + \smallfrac{11}{2} \onevector{5} =
\smallfrac{3}{2}(1,2,3,4,5)^T.\] The derivation $D$ defines a metric
solvable extension of $(\frakn_\mu,Q)$ which is Einstein.

 The Dynkin diagram $S(U),$ illustrated in Figure \ref{A4}, is that of
 $A_4$ and the structure vector defines a weighting so that vertices
 have weights 3,2,2, and 3 as shown.
\begin{figure}
\begin{center}
\setlength{\unitlength}{1mm}
\begin{picture}(70,16)
\multiput(10,8)(20,0){4}{\circle*{1.5}} \put(10,8){\line(1,0){60}}
\put(10,10){$3$} \put(30,10){$2$} \put(50,10){$2$} \put(70,10){$3$}
\put(5,3){$(1,2,3)$} \put(25,3){$(1,4,5)$} \put(45,3){$(2,3,5)$}
\put(65,3){$(1,3,4)$}
\end{picture}
\end{center}
\caption{The weighted generalized Dynkin diagram $S(U)$ for Example 12}
\label{A4}
\end{figure}
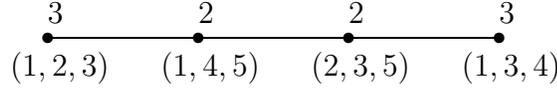
It is easy to check that Equation \eqref{weighting} holds for the
weighting of Dynkin diagram-- triple the weight of any vertex plus the
sum of the weights of the adjacent vertices has the constant value of
$11 = -2 \beta.$ There are no $(-1)$-weighted edges due to the
lack of $-1$s in the Gram matrix; this tells us that the Jacobi
condition holds.
\end{example}

The next example illustrates both Theorem \ref{mainthm} and Theorem
\ref{jacobicondition}.
\begin{example}\label{mu9} 
Let $\frakn$ be a six-dimensional vector space equipped with the inner
product $Q$ with respect to which $\calB = \{X_i\}_{i=1}^6$ is an
orthonormal basis.  Consider the level set $\Omega_\Lambda$ in
$\Lambda^2 \frakn^\ast \otimes \frakn$ for
\[ \Lambda = \{ (1,2,4), (1,4,5), (1,5,6), (2,3,5), (3,4,6) \} .\]
The Gram matrix of $\Lambda(\frakn_\mu,\calB)$ for $\mu$ in $\Lambda$
 is $\D{U =
\begin{bmatrix}  
3 & 0 & 1 & 1 & -1 \\ 0 & 3 & 0 & 1 & 1 \\ 1 & 0 & 3 & -1 & 1 \\ 1 & 1
& -1 & 3 & 1 \\ -1 & 1 & 1 & 1 & 3\\
\end{bmatrix}
}.$ A computation shows that a structure vector $[\alpha^2]$ is a
solution to $Uv=\onevector{5}$ if and only if it is of the form
\begin{equation}\label{mu9solution} [\alpha^2] = \begin{bmatrix}
(\alpha_{12}^{4})^2 \\ (\alpha_{14}^{5})^2 \\ (\alpha_{15}^{6})^2 \\
(\alpha_{23}^{5})^2 \\ (\alpha_{34}^{6})^2 \\
\end{bmatrix} = \frac{1}{20}
\begin{bmatrix} 
1 \\ 4 \\ 9 \\ 8\\ 0 \\
\end{bmatrix}
+ \frac{s}{20}
\begin{bmatrix} 
1 \\ 0 \\ -1 \\ -1 \\ 1 \\
\end{bmatrix}
\end{equation}
The only entries of minus one in the Gram matrix are those from the
inner products $\la Y_{12}^4, Y_{34}^6 \ra$ and $ \la Y_{15}^ 6 ,
Y_{23} ^ 5\ra,$ which correspond to nontrivial terms
\begin{align*} 
(\alpha_{15}^6 \alpha_{23}^5 )^2 &= \frac{1}{20^2} (9 - s) (8 - s)
\quad \text{and}\\ (\alpha_{12}^4 \alpha_{34}^6 )^2 &= \frac{1}{20^2}
(1 + s)s \\
\end{align*}
in the Jacobi identity in Equation \eqref{jacobiidentity}. Both of the
expressions are from the nontriviality of the equation with $\{i, j,
k\} = \{1, 2, 3\}$ and $m =6.$ By Theorem \ref{jacobicondition}, the
Jacobi identity holds for $(\frakn_\mu,\calB)$ with structure vector
$[\alpha^2]$ if and only if
\[ \sum_{s< 6} \alpha_{12}^s \alpha_{s3}^6 + \alpha_{23}^s \alpha_{51}^6 + 
\alpha_{31}^s \alpha_{s2}^6 = 0.\] Substituting in for the other
structure constants (all zero), we find that
\begin{align*}
 \alpha_{12}^4 \alpha_{34}^6 + \alpha_{23}^4 \alpha_{15}^6 &= 0 \\
 (\alpha_{12}^4 )^2 (\alpha_{34}^6)^2 &= (\alpha_{23}^4)^2
 (\alpha_{15}^6)^2 \label{mu9s=} \\ (9-s)(8-s) &= (1 + s)s \\ s &= 4.
\end{align*}
When we let $s=4$ in Equation \eqref{mu9solution}, we get
 \[ [\alpha^2]=\frac{1}{20}(5,4,4,5,4)^T, \]
and the resulting $(\frakn_\mu,Q)$ is a Lie algebra.  After rescaling
and solving for structure constants from $[\alpha^2]$, we find that
\[ [X_1,X_2]=\sqrt{5} X_4 \qquad [X_1,X_4] = 2 X_5 \qquad
 [X_1,X_5]=2 X_6 \] \[ [X_2,X_3]= -\sqrt{5} X_5 \qquad [X_3,X_4]=2 X_6
\] defines a nilsoliton metric Lie algebra.  We could have made other
sign choices for the structure constants as long as the signs of
$\alpha_{12}^4 \alpha_{34}^6$ and $\alpha_{23}^4 \alpha_{15}^6$ were
different; however, all these sign choices yield isometric
nilsolitons.
\end{example}

\subsection{Applications to graded metric filiform Lie algebras}
\label{derivations}  
In this section, we restate Theorem \ref{hebervar} of Section
\ref{proof of main thm} in terms of the generalized Dynkin diagrams
and give an application to filiform nilsoliton metric Lie algebras.

Suppose that an $n$-dimensional metric nilpotent Lie algebra
 $(\frakn,Q)$ admits a symmetric derivation $D$ with eigenvalues
 $\mu_1, \ldots,\mu_n.$ By Lemma 2.2 of \cite{heberinv}, $D$ commutes
 with the Ricci endomorphism of $(\frakn,Q),$ so there exists a Ricci
 eigenvector basis $\calB = \{X_i \}_{i = 1}^n$ adapted to $D;$ that
 is, each $X_i$ is an eigenvector for $D.$ By Theorem \ref{hebervar},
 $\Lambda(\frakn,\calB)$ is a subset of $\Sigma_{(\mu_1,\ldots,
 \mu_n)}.$ Conversely, if $(\frakn,Q)$ has Ricci eigenvector basis
 $\calB$ and $\Lambda(\frakn,\calB)$ is a subset of
 $\Sigma_{(\mu_1,\ldots, \mu_n)},$ then by Theorem
 \ref{nilsolitonconstraint} $D(X_i) = \mu_i X_i$ defines a derivation
 of $\frakn.$

The following theorem summarizes the previous discussion.
\begin{theorem}\label{minorgraph}  
Let $(\frakn,Q)$ be a $n$-dimensional metric nilpotent Lie algebra
with Ricci eigenvector basis $\calB.$ Suppose that $D$ is a symmetric
derivation of $\frakn$ with eigenvalue vector $v_D=(\mu_1, \ldots,
\mu_n).$ Then the generalized Dynkin diagram for
$\Lambda(\frakn,\calB)$ is a minor graph of the generalized Dynkin
diagram of $\Sigma_{\mu_1, \ldots, \mu_n},$ which in turn is a subset
of $\Theta_n.$ Conversely, if $\Lambda(\frakn,\calB)$ is a minor graph
graph of $\Sigma_{\mu_1, \ldots, \mu_N}$ then the endomorphism $D$ of
$\frakn$ defined by $D(X_i) = \mu_i X_i$ for $i = 1, \ldots, n$ is a
symmetric derivation $D$ of $\frakn.$
\end{theorem}
 For example, the characteristically nilpotent Lie algebra defined by
Dixmier and Lister in Example \ref{dixmier-lister} does not admit
semisimple derivation, so its generalized Dynkin diagram $S(\Lambda)$
can not be a minor graph of $S(\Sigma_{(\mu_1, \ldots, \mu_n)})$ for
any $(\mu_1, \ldots, \mu_n)$ in $\boldR^n.$

The study of the graphs $\Sigma_{(\mu_1, \ldots, \mu_n)}$ and
$\Theta_n$ has two applications to nilsolitons.  First, one can show
that a generalized Dynkin diagram $S(U)$ is realized as the
generalized Dynkin diagram for a nilsoliton metric Lie algebras by
finding $S(U)$ as a minor graph of some $S(\Sigma_{(\mu_1, \ldots,
\mu_n)}),$ finding a nilsoliton weighting on $S(U)$ by solving $Uv =
\boldone$, defining an algebra using structure constants from the
weighting, and checking the Jacobi identity with Theorem
\ref{jacobicondition}.  Second, one can look for graph-theoretic or
combinatorial obstructions for a metric nilpotent Lie algebra to admit
a symmetric derivation by finding properties of the graphs
$S(\Sigma_{(\mu_1, \ldots, \mu_n)})$ that give restrictions on the
types of minor graphs they can have.

In the next example we describe the generalized Dynkin diagrams for
the sets $\Sigma_n := \Sigma_{(1,2,\ldots, n)}$ with multiplicities
zero or one.
\begin{example}\label{Sigma_n} 
 Let $(\frakn,Q)$ be an inner product space of dimension $n$ with
orthonormal basis $\calB = \{X_i\}_{i = 1}^n.$ Let $\Sigma_n$ be the
set of integer triples defined by
\begin{align*} 
\Sigma_n &= \Sigma_{(1,2,\ldots,n)} \\
&= \{(i,j,i + j) \, | \, 1 \le i < j  \le n, i + j \le n \}.\end{align*}
Let $U_{\Sigma_n}$ denote the Gram matrix for $\Sigma_n$ and let
 $S(\Sigma_n) = S(U_{\Sigma_n})$ denote the generalized Dynkin diagram
 associated to $U_{\Sigma_n}.$

It is easy to check that $S(\Sigma_3)$ is the Dynkin diagram of type
$A_1,$ $S(\Sigma_4)$ is the Dynkin diagram of type $A_1 \oplus A_1,$
and $S(\Sigma_5)$ is the Dynkin diagram of type $A_4.$ When $n \ge 6,$
the Gram matrix always has entries of minus one coming from triples of
the form $(i,j,k)$ and $(l,k,m)$ with $l \ne, i, l \ne j$ in
$\Sigma_n$ such as $(1,3,4)$ and $(2,4,6)$ in $\Sigma_6.$

The generalized Dynkin diagram for $\Sigma_6$ is in Figure
\ref{Sigma6}.  It was shown in \cite{mcdanielpayne05} that $\Delta_6$
admits a one-parameter family of nilsoliton weightings with $\beta =
-2,$ and the Jacobi condition from the two dotted edges as described
by Theorem \ref{jacobicondition} gives a constraint so that up to
scaling, there is a unique nilsoliton metric $(\frakn,Q)$ with
$\Lambda(\frakn,\calB)= \Sigma_6$ for a Ricci eigenvector basis
$\calB.$
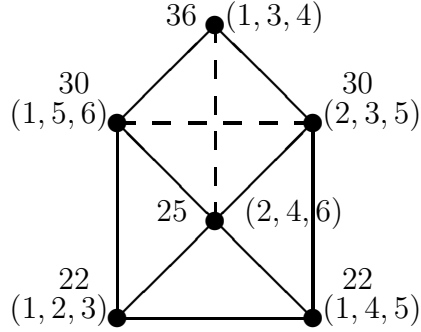
\begin{figure}
\begin{center}
\setlength{\unitlength}{1.3mm}
\begin{picture}(20,30)
\thicklines \put(0,0){\circle*{2}} \put(-11,0){$(1,2,3)$}
\put(-6,3){$22$} \put(20,0){\circle*{2}} \put(23,3){$22$}
\put(21,0){$(1,4,5)$} \put(0,20){\circle*{2}} \put(-11,20){$(1,5,6)$}
\put(-6,23){$30$} \put(20,20){\circle*{2}} \put(21,20){$(2,3,5)$}
\put(23,23){$30$} \put(10,10){\circle*{2}} \put(13,10){$(2,4,6)$}
\put(4,10){$25$} \put(10,30){\circle*{2}} \put(11,30){$(1,3,4)$}
\put(5,30){$36$} \put(0,0){\line(1,0){20}}
\dashline[+30]{2}(10,10)(10,30) \put(0,0){\line(0,1){20}}
\put(0,0){\line(1,1){20}} \put(20,0){\line(-1,1){20}}
\dashline[+30]{2}(0,20)(20,20) \put(20,0){\line(0,1){20}}
\put(0,20){\line(1,1){10}} \put(20,20){\line(-1,1){10}}
\end{picture}\end{center}
\caption{The generalized Dynkin diagram $S(U)$ for $\Sigma_6$ with a
nilsoliton weighting }\label{Sigma6} 
\end{figure}
\end{example}

\section{The solution space to $Uv=\boldone$}\label{solutions}
There are four distinct possibilities
 for the solution space of a linear system $Uv=-2\beta \onevector{m}$
 arising from Theorem \ref{mainthm}.  Solution spaces of the first
 four of the possible types occur: the first type in Example
 \ref{concreteexample1}, the second types and third types 
from horospheres for a quaternionic
 hyperbolic space and  the Cayley plane  in Examples \ref{quaternionic}
and \ref{cayley}, 
and the fourth type in
 Example \ref{dixmier-lister}.  As symmetric Cartan matrices of finite
 and affine type have been classified, it is not hard to list all
 systems of finite and affine type and to find solutions to $Av = \mu
 \boldone$ in each case. See page 54 of \cite{kac90} for the solutions
 in the affine case, presented as weighted Dynkin diagrams.

\bigskip

\noindent{\em Proof of Theorem \ref{solsp}.}  Substituting $U = cA + d
\onematrix{m}{m}$ into $Uv = \lambda \onevector{m}$ and using the fact
that $\onematrix{m}{m} v = (\onevector{m} \cdot v) \onevector{m},$ we
derive the equivalent equation
\begin{equation}\label{justabove}
Av = \frac{\lambda - d (\onevector{m} \cdot v) }{c} \, \onevector{m}.
\end{equation}
Viewing $\lambda$ and $\mu$ as parameters, the two families of linear
equations $Uv = \lambda \onevector{m}$ and $Av = \mu \onevector{m} $
have solutions that are in bijective correspondence: if $v$ is a
solution to $Uv = \lambda \onevector{m},$ then $v$ is a solution to
$Av = \mu \onevector{m} $ with parameter
\[ \mu = M(\lambda,v) =  \frac{\lambda - d (\onevector{m} \cdot v)  }{c},\]  
and if $v$ is a solution to $Av = \mu \onevector{m},$ then $v$ is a
solution to $U v = \lambda \onevector{m}$ with parameter
\[ \lambda=  L(\mu,v) =    c \mu + d (\onevector{m} \cdot v)   . \]
If $Uv=\lambda\onevector{m},$ then $L(M(\lambda,v),v) = \lambda$ and
if $Av = \mu \onevector{m},$ then $M(L(\mu,v),v)=\mu.$

By Theorem \ref{kac4-3}, there are three mutually exclusive cases for
the generalized Cartan matrix $A$: it is of finite type, affine type,
or indefinite type.  We consider each case in turn.

First suppose that $A$ is of finite type.  By Theorem \ref{kac4-3},
$A$ is invertible, so the equation $Av = \onevector{m}$ has a unique
solution $v_0,$ and as $Av_0 = \onevector{m}$ is positive, $v_0$ is
positive.  As the vector $v_0$ is a solution to $Av = \mu
\onevector{m}$ with $\mu = 1,$ it is a solution to $Uv = \lambda
\onevector{m}$ with $\lambda = L(1,v_0) = c + d (\onevector{m} \cdot
v_0).$ The positivity of $c$ and $d$ and the fact that $v_0 > 0$ force
$L(1,v_0)$ to be positive.  Rescaling gives a solution $-
\frac{2\beta}{L(1,v_0)} v_0$ to $U v = -2\beta \onevector{m}.$ By
hypotheses $\beta < 0,$ so $-\frac{2\beta}{L(1,v_0)} v_0$ is a
positive vector.  It is the unique solution to $U v = -2\beta
\onevector{m}$; were it not unique, there would exist a vector $v$ in
the solution space, not a multiple of $v_0,$ and sufficiently close to
$v_0$ that $M(-2\beta,v) \ne 0$ ($M(-2\beta,v_0) = 1$).  Then
$\frac{1}{M(-2\beta,v)} v$ would be a second solution to $Av=
\onevector{m}$, contradicting the nonsingularity of $A.$

Suppose that $A$ is of affine type.  By Theorem \ref{kac4-3}, there
exists positive $v_0$ with $Av_0 = 0.$ Then $v_0$ is a solution to $Av
= \mu \onevector{m}$ with $\mu = 0,$ making $v_0$ a solution to $Uv =
\lambda \onevector{m}$ with $\lambda = L(0,v_0) = d (\onevector{m}
\cdot v) > 0.$ Rescaling to $v = -\frac{2cv\beta}{L(0,v_0)} v_0$ gives
a positive solution $v$ to Equation $Uv = -2\beta \onevector{m}.$ Now
we show that the solution $v_0$ is unique.  Suppose that $v_1$ is
another solution to $Uv = -2\beta \onevector{m}.$ Then $v_1$ is a
solution to $Av = \mu \onevector{m}$ with $\mu = M(-2\beta,v_1).$ If
$\mu \ge 0,$ then $Av_1 \ge 0,$ and Theorem \ref{kac4-3} implies that
actually $Av_1 =0.$ If $\mu \le 0;$ then $A(-v_1) \ge 0,$ by again by
Theorem \ref{kac4-3}, $A(-v_1)=0.$ Thus $v_1$ lies in the
one-dimensional kernel of $A$ and must be a multiple $k$ of $v_0.$
Then $U(v_0) = -2\beta \onevector{m} = U(k v_0)$ and the nonzero
vector $U(v_0)$ is a multiple $k$ of itself.  Hence $k=1$ and $v_0 =
v_1.$

Now suppose that $A$ is of indefinite type.  We must show that if $Uv
= -2\beta \onevector{m}$ is consistent and there exists a positive
solution, then Case (Ind1) holds.  First we assume that $Uv = -2\beta
\onevector{m}$ is consistent, and we show that $\ker U \subset \ker
A.$ Recall that a matrix equation $Bx=y$ is consistent if and only if
$y$ is in $\range B = (\ker B^T)^\perp,$ so $\ker U \subset
\onevector{m}^\perp.$ Suppose that $v \in \ker U.$ Then $v$ is a
solution to $Uv = \lambda \onevector{m}$ with $\lambda = 0.$ The
one-to-one correspondence between solution spaces makes $v_0$ a
solution to $Av = \mu \onevector{m}$ with $\mu = M(0,v_0) = 0.$ Thus,
$v \in \ker A.$

Let $v_0$ be a fixed positive solution to $Uv = -2\beta
\onevector{m}.$ The bijective correspondence of solution spaces makes
$v_0$ a solution to $Av = \mu \onevector{m}$ with $\mu =
M(-2\beta,v_0).$ The complete solution space to $Av = \mu
\onevector{m}$ with $\mu = M(-2\beta,v_0)$ is $\{ v_0 + w \, | \, w
\in \ker A \}.$

Any vector $v_0 + w$ in this set is a solution to $Uv = \lambda
\onevector{m}$ with $\lambda = L(M(-2\beta,v_0),v_0 + w).$ Since
$L(M(-2\beta,v_0),v_0) = -2\beta > 0,$ by continuity, for $w$ in a
small neighborhood $U$ of the zero vector in $\ker A$,
$M(L(-2\beta,v_0),v_0 + w) > 0.$ Rescaling, we find that the solution
space to the system $Uv = -2\beta \onevector{m}$ contains
\[ \left\{ \frac{-2\beta}{L(M(-2\beta,v_0),v_0 + w)}(v_0 + w)  \, 
| \, w \in U \right\}.\] This is a space of the same dimension as
 $\ker A.$ Hence, $\dim \ker A \le \dim \ker U,$ so $\ker A = \ker U.$
 
In the indefinite case, the fact that $\ker A \subset
\onevector{m}^\perp$ implies that the set of nonnegative solutions is
bounded and forms a simplex of dimension $\ker A$.

Now assume that $U$ is the Gram matrix for a metric nilpotent Lie
algebra $(\frakn,Q)$ with respect to some basis $\calB.$ Let $Y$ be
the $m \times n$ root matrix for $\Lambda(\frakn,\calB).$ By
definition, $U = Y^T Y.$ We will show that the kernel of the symmetric
matrix $U$ is contained in $\onevector{m}^\perp.$ Suppose that $v$ is
in the kernel of $U.$ The kernel of $U$ is equal to the kernel of
$Y^T;$ hence $v$ is in the kernel of $Y^T.$ Then
\[  \onerowvector{n} Y^T v =  
\onerowvector{n} \zerovector{n} = \zerovector{1}.
\]
By definition, the rows of $Y$ are the structure vectors $Y_{ij}^k =
\varepsilon_i^T + \varepsilon_j^T - \varepsilon_k^T$ for $(i,j,k)$ in
$\Lambda.$ The product of the row vector $\onerowvector{n}$ and any
column of $Y^T$ is one, so $\onerowvector{n} Y^T = \onerowvector{m}.$
Substituting, we conclude that $\onerowvector{m} v = \zerovector{1}$
as desired. \qed

\bigskip

\section{All soliton metrics in a class $\Omega_\Lambda$}\label{proof of 
allnilsolitons} Here is the proof of Theorem \ref{allnilsolitons}.
\begin{proof}
Fix a subset $\Lambda$ of $\Theta_n$ and the root matrix $Y$ and Gram
matrix $U$ for $\Lambda.$ As $\Lambda \subset \Theta_n,$ the matrix $A
= 2U - 4\onematrix{m}{m}$ is a generalized Cartan matrix by Theorem
\ref{gramproperties}.

Suppose that $U$ is indecomposable.  Theorem \ref{solsp} then applies
 with $U = \half A + 2 \onematrix{m}{m}$.  In Cases (Fin) and (Aff) of
 the theorem, there is a unique positive solution $[\alpha^2]$ to $Uv
 = -2\beta \onevector{m}.$ There are $2^m$ different $\mu$ with this
 structure vector, one for each choice of sign in each entry of
 $[\alpha^2].$ If we are in Case (Ind1) of Theorem \ref{solsp}, $U$ is
 of indefinite type and there exists a solution $v_0$ to
 $Uv=\onevector{m}$ with positive entries.  The set of solutions to
 $Uv=\onevector{m}$ in the cone $\{ v \in \boldR^m \, | \, v \ge 0 \}$
 is a simplex of the same dimension as $\ker A.$ Taking square roots
 entry-wise and allowing all sign choices gives $2^m$ simplices of
 $\mu$ so that $(\frakn_\mu,Q)$ satisfies $U[\alpha^2] = -2\beta
 \onevector{m}$.  Only the interiors of these simplices lie inside
 $\Omega_\Lambda.$ In Case (Ind3) of Theorem \ref{solsp} there are no
 $\mu$ in $\Omega_\Lambda$ whose structure vector is a solution to $Uv
 = -2\beta \onevector{m}.$

By Equation \eqref{riccivector}, if $(\frakn_\mu,Q)$ has structure
vector $v_0+v,$ with $v$ in $\ker A = \ker Y^T,$ then the Ricci vector
is
\[ \bfRic_\mu^\calB = -\frac{1}{2}Y^T (v_0 + v) = -\frac{1}{2} Y^T  v_0. \]
In other words, for all metric algebras $(\frakn_\mu,Q)$ with
structure vector $v_0 + v$ that satisfy the nilsoliton condition, the
Ricci vector is the same.  Since $\calB$ is a common Ricci eigenvector
basis for all such $(\frakn_\mu,Q)$, once the Ricci vector is
constant, all the nil-Ricci endomorphisms agree.

If $U$ is decomposable, after permuting the basis vectors, it can be
written as the block diagonal sum of indecomposable Gram matrices
$U_1, U_2, \ldots, U_r.$ With respect to the new basis, the vector
$\onevector{m}$ is still written as $\onevector{m}.$ The solution
space to $U v = \onevector{m}$ is the product of the solution spaces
to the systems $U_i v = \onevector{m_i},$ where $U_i$ is $m_i \times
m_i$ and $i = 1, \ldots,r.$ If $U$ is nonsingular, then each $U_i$ is
nonsingular with unique solution $v_i$ to $U_i v = \onevector{m_i},$
making $v = v_1 \times \ldots \times v_r$ the unique solution to $Uv =
\onevector{m}.$ If $U$ is nonsingular and there is some $i$ so that
$U_i v = \onevector{m_i}$ has no solution with positive entries, then
$Uv = \onevector{m}$ has no solution with positive entries.  If $U$ is
nonsingular and $U_i v = \onevector{m_i}$ has positive solutions for
all $i,$ then the solution space is the nontrivial product of the
solution spaces to the equations $U_i v = \onevector{m_i}.$ The
product of simplices is a simplex and the dimension of a product is
the sum of the dimensions of the factors, so the theorem holds for the
decomposable case.  When $U$ is decomposable, and $(\frakn,Q)$ is
nilsoliton with nilsoliton constant $\beta,$ then $(\frakn_i,Q_i)$ is
also nilsoliton with nilsoliton constant $\beta,$ where $\frakn_i$ is
the subspace for the submatrix $U_i$ and $Q_i$ is the restriction of
$Q$ to that subspace.  The Ricci vector for $(\frakn,Q)$ is obtained
by concatenating the Ricci vectors $\bfRic_i$ for $(\frakn_i,Q_i).$
Thus, the constancy of the Ricci vector for nilsoliton metrics in the
same equivalence class holds in the decomposable case also.   
\end{proof}

As a corollary to Theorem \ref{allnilsolitons}, we find that any time
the number of nontrivial structure constants $\alpha_{ij}^k$ with $i <
j$ exceeds the dimension of a nilsoliton $(\frakn_\mu,Q),$ there are
continuous families of metric algebras around $(\frakn_\mu,Q)$ also
satisfying the nilsoliton condition.  The Jacobi identity may prohibit
members of this continuous family from being nilsoliton metric Lie
algebras.
\begin{corollary}\label{m>n} 
Let $(\frakn,Q)$ an $n$-dimensional inner product space and let
$\calB$ be a orthonormal basis for $\frakn$.  Let $\Omega_\Lambda$ be
a nontrivial level set for $\mu \mapsto \Lambda$ so that $\Lambda
\subset \Theta_n.$ and $ \| \Lambda \| > n.$ Let $\beta < 0.$ Then the
Gram matrix $U$ for $\Omega$ is singular, and either $\Omega_\Lambda$
contains no $\mu$ satisfying $Uv = -\beta \onevector{m}$ or there is a
continuous family of dimension at least $m - n$ of $\mu$ in
$\Omega_\Lambda$ so that $(\frakn_\mu,Q)$ satisfies $Uv = -\beta
\onevector{m}.$
\end{corollary}

\begin{proof} 
The rank of $U$ is equal to the number of independent root vectors.
The root vectors are $1 \times n$ in size, so the maximal rank of $U$
is $n.$ By Theorem \ref{allnilsolitons}, the set of $\mu$ in
$\Omega_\Lambda$ that satisfy $U [\alpha^2] = -2\beta \onevector{m}$
is either empty or of dimension $\ker U = m - \rank U \ge m - n ,$
where $m = \| \Lambda\|.$   \end{proof}

 In Example \ref{cayley} we
 will have $(m,n)=(28,15)$ and an eight-parameter subfamily of
 solutions.  The  nilpotent metric Lie algebra
 $(\frakn,Q)$ for quaternionic space $H^n{\boldH}$
of dimension $4n$ has $(m,n)=(6n,4n+3)$ in its most natural 
representation (as in Example \ref{quaternionic}), therefore for $n \ge 2$
there will be continuous families of nilsolitons around $(\frakn,Q)$
for $n \ge 2.$ 
In Example \ref{quaternionic}, we will have $(m,n)=(12,11)$ and a
 two-parameter subfamily of solutions. 
A similar computation for $H^3(\boldH)$ not included here 
yields a four-parameter family
of Einstein solvmanifolds around quaternionic hyperbolic space of
dimension sixteen.

\section{Examples}\label{examples}

\subsection{A characteristically  nilpotent example}
The next example is of Case (Ind2) in Theorem \ref{solsp}, when
  solutions to $Uv = \boldone$ exist, but the solution vectors never
  have all entries positive.  The example, due to J.\ Dixmier and W.\
  G.\ Lister, was the first example of a nilpotent Lie algebra
  with a nilpotent derivation algebra.  A Lie algebra $\frakg$ so that
  the Lie algebra of derivations of $\frakg$ is nilpotent is called
  {\em characteristically nilpotent} (see \cite{khakimdjanov02} for a
  survey).  Such a Lie algebra admits no semisimple derivations, and
  hence can not admit a nilsoliton metric.  All nilpotent structures
  of dimension six or less admit nilsoliton metrics (see
  \cite{lauret02} and \cite{will03} for a classification), and the
associated derivation $\Ric - \beta \Id$ is nontrivial and 
semisimple.  The lowest dimension in which 
  characteristically nilpotent Lie algebras occur is
  (\cite{favre72}).
\begin{example}\label{dixmier-lister} (\cite{dixmierlister})
 Let $(\frakn,Q)$ be the metric Lie algebra with orthonormal basis
$\calB = \{X_i\}_{i=1}^8$ and nontrivial structure constants encoded
in the Lie brackets
\begin{alignat*}{5}
[X_1,X_2] &= X_5 \qquad &[X_1,X_3] &= X_6 \qquad &[X_1,X_4] &= X_7 \\
[X_1,X_5] &= -X_8 \qquad &[X_2,X_3] &= X_8 \qquad &[X_2,X_4] &= X_6 \\
[X_2,X_6] &= -X_7 \qquad &[X_3,X_4] &= -X_5 \qquad &[X_3,X_5] &= -X_7
\\ \, & \, &[X_4,X_6] &= -X_8 & \, & \, 
\end{alignat*}
By Corollary \ref{orthogonal}, the basis $\calB$ is a Ricci
eigenvector basis.  We find using MAPLE that $U$ is singular with the
vectors \begin{align*} v_1 &= (1,-1,0,0,0,0,-1,0,1,0)^T \qquad
\text{and} \\ v_2 &= (0,1,0,-1,0,0,0,-1,0,1)^T \end{align*} spanning
$\ker U = \ker A = \ker Y^T.$ The general solution to $Uv =
\onevector{10}$ is
\[ v= 
(-\half,\smallfrac{5}{2},-7,6,-7,
\smallfrac{5}{2},6,\smallfrac{11}{2},0,0)^T + s v_1 + t v_2;\] none of
these solutions satisfy $v > 0,$ as the third coordinate is always
negative, so $A$ is of
indefinite type by Theorem \ref{kac4-3}.  Note that $v_1$ and $v_2$
are orthogonal to $\onevector{10}$ since $Uv = \onevector{10}$ and
$Av=\onevector{10}$ are consistent.
\end{example}

\subsection{The class $L_n$ of filiform Lie algebras}\label{Lnsubsection}

Let $\frakg$ be a Lie algebra.  The descending central series of
$\frakg$ is defined by $\frakg^{(1)} = \frakg$ and $\frakg^{(j + 1)} =
[\frakg, \frakg^{(j)}]$ for $j > 1.$ The Lie algebra $\frakg$ is
nilpotent if and only if there is an integer $k$ so that
$\frakg^{(k)}$ is trivial.  If $k$ is the smallest integer so that
$\frakg^{(k + 1)}$ is trivial, then $\frakg$ is said to be $k$-{\em
step} nilpotent.

 A $k$-step nilpotent Lie algebra of dimension $n$ is called {\em
filiform} if $k + 1 = n.$ See \cite{khakimdjanov02} for a survey of
filiform and characteristically nilpotent Lie algebras.  The
$(n+1)$-dimensional filiform nilpotent Lie algebra $L_n$ is defined
with respect to the basis $\{X_i\}_{i=0}^n$ by the bracket relations
\[
[X_0,X_i] = X_{i+1} \qquad \text{\rm for \,} i=1,\ldots,n-1 .\] J.\
Lauret has shown using variational methods that every Lie algebra in
this family admits an inner product $Q$ so that $(L_n,Q)$ is a
nilsoliton metric Lie algebra.  We give an alternate algebraic proof
of this fact.

\begin{theorem}[\cite{lauret02}]\label{Ln} For all $n,$ the filiform 
nilpotent Lie algebra $L_n$ admits a unique nilsoliton metric.
\end{theorem}

\begin{proof} 
Take the inner product $Q$ on the vector space $\frakn = \myspan \{X_i
\}_{i=0}^n$ with respect to which $\calB = \{X_i\}_{i = 0}^n$ is
orthonormal.  We want to find nonzero values of $a_i$ such that the
bracket relations
\begin{equation}\label{newLn}
[X_0,X_i] = a_iX_{i+1} \qquad \text{\rm for \,} i=1,\ldots,n-1
\end{equation} define a nilpotent structure $\mu$ on $\frakn$ so that
$(\frakn_\mu, Q)$ is nilsoliton.  Writing the bracket relations with
respect to the basis $X_0^\prime= X_0, X_1^\prime = X_1, X_i^\prime =
\left( \Pi_{i = 1}^{i - 1} a_i \right) X_i$ for $i =2$ to $n$ demonstrates
that such a Lie algebra $\frakn_\mu$ is isomorphic to $L_n.$

By Corollary \ref{orthogonal}, $\calB$ is a Ricci eigenvector basis.
The Gram matrix $U$ for $(\frakn_\mu,Q)$ and $\calB$ and the
associated Cartan matrix $A = U - \onematrix{(n-1)}{(n-1)}$ are given
by
\[ U = \begin{bmatrix} 3 & 0 & 1  & \cdots & 1 \\
0 & 3 & 0 & \cdots & 1\\ 1 & 0 & 3 & \cdots & 1 \\ \vdots & \vdots &
\vdots & \ddots & \vdots \\ 1 & 1 & 1 & \cdots & 3
\end{bmatrix} \qquad \text{\rm and} \qquad  
A = \begin{bmatrix} 2 & -1 & 0 & \cdots & 0 \\ -1 & 2 & -1 & \cdots &
  0\\ 0 & -1 & 2 & \cdots & 0 \\ \vdots & \vdots & \vdots & \ddots &
  \vdots \\ 0 & 0 & 0 & \cdots & 2
\end{bmatrix}.
\] 
By Theorem \ref{mainthm}, $(\frakn_\mu,Q)$ is nilsoliton if and only
if $v = [\alpha^2]$ is a solution of $Uv = -2\beta \onevector{(n-1)}.$
Note that $A$ is the Cartan matrix of type $A_{n-1},$ so by Theorem
\ref{solsp}, the equation $U v = -2\beta \onevector{(n-1)}$ has a
unique positive solution for any $\beta < 0.$ By Theorem \ref{solsp},
these solutions are in bijective correspondence with solutions to the
systems $A v = \nu \onevector{(n-1)}$ for some $\nu > 0.$

Writing out the equations of the system $A v = \nu \onevector{(n-1)}$,
we find the recursion relation
\[  -v_{i - 1}+ 2v_i -  v_{i+1} = \nu \qquad (v_0 = v_{n} := 0)\]
which has solution $v_i = \frac{\nu}{2}i(n-i ).$ The solutions $v_i$
are positive for all $i = 1, \ldots, n-1,$ so we may let
\[ a_i = \sqrt{v_i}  = \sqrt{\frac{\nu}{2}i(n-i) }\] for $i=1,
\ldots, n-1$ to define a nilsoliton metric Lie algebra $(\frakn_\mu,
Q).$ Note that $a_i = a_{n-i};$ this will be used in the proof of
Theorem \ref{Qn}.  Normalizing $\nu$ to be equal to two, and using the
results of Example \ref{Ln example, part one}, we find that
\begin{align*} 
\ric(X_0,X_0) &= -\frac{1}{2} \sum_{i = 1}^{n-1} a_i^2 - a_{i-1}^2 =
-\frac{1}{2} \left[ n \left( \sum_{i = 1}^{n-1} i \right) - \sum_{i =
    1}^{n-1} i^2 \right] \\ &= -\smallfrac{1}{12}(n-1)n(n+1), \\
\ric(X_i,X_i) & = -\half(n+1) + i \text{\, for $1 \le i \le n,$}
\end{align*}
and that the nilsoliton constant $\beta$ is $
-\smallfrac{1}{12}(n-1)n(n+1) - 1.$

Uniqueness follows from the uniqueness of soliton metrics on nilpotent
Lie groups (\cite{lauret01a}).   \end{proof}

\subsection{Symmetric examples}

In Theorem J of (\cite{heberinv}), J.\ Heber showed that real and
complex hyperbolic spaces are isolated among Einstein solvmanifolds
and he predicted the existence of continuous families of Einstein
solvmanifolds around quaternionic hyperbolic space $H^m(\boldH)$ for
$m\ge 2$ and around the Cayley plane.  These families yield
continuous families of soliton nilmanifolds around horospheres
for the symmetric spaces.  
In this section we consider 
nilmanifolds associated to horospheres of 
rank one symmetric spaces of noncompact type, 
and we the find subfamilies of the families of soliton nilmanifolds 
associated to the Einstein solvmanifolds that
Heber predicted.  We also consider a horosphere for 
a higher rank symmetric space and find continuous families of soliton
metrics around it. 
  
  First we consider the the
Heisenberg algebra.
\begin{example}\label{heisenberg} 
Let $\frakh_n$ be the $(2n+1)$-dimensional Heisenberg algebra. The
 algebra $\frakh_n$ may be represented with basis \[ \calB = \{X_1,
 \ldots, X_n, Y_1, \ldots, Y_n, Z\} \] and nontrivial structure
 constants coming from the bracket relations $[X_i,Y_i] = Z$ for $i=1,
 \ldots n.$ We define the inner product $Q$ on $\frakh_n$ so that
 $\calB$ is orthonormal.  It is a Ricci eigenvector basis by Corollary
 \ref{orthogonal}.

The root vectors for $(\frakh_n,Q)$ with respect to $\calB$ are
linearly independent, and the $n \times n$ Gram matrix $U$ has $u_{ii}
= 3$ for $i = 1, \ldots, n$ and $u_{ij} = 1$ for $i \ne j.$ The graph
$S(U)$ is the complete graph on $n$ vertices, and if $A = U -
\onematrix{n}{n},$ the graph $S(A)$ is $n$ vertices with no edges.
The graph is regular so a constant weighting assigning the value of
one to each vertex satisfies Equation \eqref{weighting}, and hence
$(\frakh_n,Q)$ is a nilsoliton metric Lie algebra.  Because $A = U -
\onematrix{(2n+1)}{(2n+1)} = 2 \Id$ is nonsingular, by Theorem
\ref{solsp}, the weighting is unique.  No matter what sign choice we
make ($\pm 1$) for the structure constants, we get an algebra
isomorphic to $\frakh_n.$ Thus, $(\frakh_n,Q)$ is the only nilsoliton
metric Lie algebra inside its equivalence class $\Omega_\Lambda.$
\end{example}

  The nilmanifold corresponding to the metric Lie algebra
$(\frakh_n,Q)$ in the previous example is isometric to any horosphere
for complex hyperbolic space of real dimension $2n+2,$ a rank one
symmetric space of noncompact type.
 In the next example, we consider
a horosphere for quaternionic symmetric space.  C.\
Gordon and M.\ Kerr presented other continuous families around
quaternionic hyperbolic spaces in \cite{gordonkerr01} and in
\cite{kerr03}.
\begin{example}\label{quaternionic}  
Let $\frakv = \boldR^{8}$ and $\frakz = \boldR^3,$ and let
\[ \calB = \{V_1, iV_1, jV_1, kV_1, 
V_2, iV_2, jV_2, kV_2, iZ, jZ, kZ\} \] be a orthonormal basis for the
vector space $\frakn = \frakv \oplus \frakz$ in the obvious way.  Let
$\mu$ be an element of $\Lambda^2 \frakn^\ast \otimes \frakn$ with
nontrivial structure constants coming from the twelve bracket
relations
\begin{alignat*}{2}
[V_s,iV_s] &= \alpha_{1i}^i(s) iZ \qquad [jV_s,kV_s] =&
\alpha_{jk}^i(s)iZ \\ [V_s,jV_s] &= \alpha_{1j}^j(s)jZ \qquad
[iV_s,kV_s] =& -\alpha_{ik}^j(s)jZ \\ \qquad [V_s,kV_s]
&=\alpha_{1k}^k(s)kZ \qquad [iV_s,jV_s] =& \alpha_{ij}^k(s)kZ
\end{alignat*}
($s=1,2$).  By Corollary \ref{orthogonal}, $\calB$ is a Ricci
eigenvector basis for all possible values of these structure
constants.   After writing out the matrices $U$ and $A$ and using
Maple to solve $Uv = \onevector{12},$ we find a four parameter family
of solutions for squares of structure coefficients.
\begin{alignat*}{3}
\alpha_{1i}^i(1)^2 &= \alpha_{jk}^i(1)^2 &= \smallfrac{1}{10}& + s + t
\\ \alpha_{1j}^j(1)^2 &= \alpha_{ik}^j(1)^2 &= \smallfrac{1}{10}& - s
\\ \alpha_{1k}^k(1)^2 &= \alpha_{ij}^k(1)^2 &= \smallfrac{1}{10}& - t
\\ \alpha_{1i}^i(2)^2 &= \alpha_{jk}^i(2)^2 &= \smallfrac{1}{10}& - (s
+ t) \\ \alpha_{1j}^j(2)^2 &= \alpha_{ik}^j(2)^2 &= \smallfrac{1}{10}&
+ s \\ \alpha_{1k}^k(2)^2 &= \alpha_{ij}^k(2)^2 &= \smallfrac{1}{10}&
+ t
\end{alignat*}
 Notice that $\frakn_\mu$ is always two-step, so  the Jacobi identity
holds trivially for  any 
$\mu$ defined by any choice of structure constants.

It can be checked that the graph $S(U)$ has as its symmetry group the
dihedral group $D_6;$ therefore the solution space shares this
symmetry.  The Lie algebra structure for the horosphere of
quaternionic hyperbolic space of dimension twelve comes from setting
all parameter values equal to zero.  Solutions are positive when $|s|,
|t|$ and $|s+t|$ are less than $1/10.$ It may be checked that the
eigenvalues for $\Ric$ are $-3/20$ on $\frakv$ and $4/20$ on $\frakz$
and that $\beta = -1/2,$ so $D$ has eigenvalues $7/20$ and $14/20$ on
$\frakv$ and $\frakz$ respectively.  The Ricci endomorphism and
$D$ are the same for all the nilsoliton Lie algebras in
$\Omega_\Lambda,$ so  
the Einstein solvable extensions all have the same Ricci endomorphism.

By considering the values $\trace J_W^2,$ for unit 
$W \in \frakz,$ it is easy to 
see that varying $s$ and $t$ so that  $0 \le s \le t \le 1/10$ and
$s + t \le 1/10,$ 
defines a family of metric  Lie algebras that 
are not pairwise  isometrically isomorphic; hence, by \cite{wilson82}
the corresponding nilmanifolds are not isometric. Thus we
have found a two-parameter family of soliton metric Lie algebras.
\end{example}

J.\ Heber showed that there exists an 84-parameter family of Einstein
solvmanifolds around the Cayley plane with the symmetric metric.  Here
we analyze the system $Uv = [1]$ that would lead to
 members of that family having
same set of nonzero structure constants relative to a natural basis.
\begin{example}\label{cayley}  
Let $\frakv = \boldR^8$ and $\frakz = \boldR^7,$ and let
\[ \calB = \{ V_0, V_1, V_2, V_3, V_4, V_5, V_6, V_7, Z_1, Z_2, Z_3,
Z_4, Z_5, Z_6, Z_7 \} \] be an orthonormal basis for $\frakn = \frakv
 \oplus \frakz$ in the obvious way.  Consider a Lie bracket $\mu$ with
 nontrivial structure constants coming from the bracket relations
\[ [V_i,V_j] =  \alpha_{ij}^k Z_k \]
where $(i,j,k)$ ranges over the set
\begin{multline*}
 \{ (0,1,1), (2,4,1), (3,7,1), (5,6,1), (0,2,2), (1,4,2), (3,5,2), \\
(6,7,2), (0,3,3), (1,7,3), (2,5,3), (4,6,3), (0,4,4), (1,2,4),\\
(3,6,4), (5,7,4), (0,5,5), (1,6,5), (2,3,5), (4,7,5), (0,6,6), \\
(1,5,6), (2,7,6), (3,4,6), (0,7,7), (1,3,7), (2,6,7), (4,5,7) \}
\end{multline*}
By Corollary \ref{orthogonal}, $\calB$ is a Ricci eigenvector basis
for all metric Lie algebra structures of this form.  When all the
structure constants are equal, $\frakn_\mu$ is the Lie algebra of $N$
from the Iwasawa decomposition $KAN$ of the isometry group of the Cayley
plane.  After finding the root vectors, computing the Gram matrix $U,$
and using MAPLE to solve $U [\alpha^2] = \onevector{28}$ we get an
eight-dimensional set of solutions
\begin{align*}
 (\alpha_{01}^1)^2 &= \smallfrac{1}{18} + a + b +c + d + 2e \\
(\alpha_{24}^1)^2 &= \smallfrac{1}{18} + a + b - c - d - 2e \\
(\alpha_{37}^1)^2 = (\alpha_{56}^1)^2 &= \smallfrac{1}{18} - a - b \\
 (\alpha_{02}^2)^2 &= \smallfrac{1}{18} -a- 2e + f + h \\
 (\alpha_{14}^2)^2 &= \smallfrac{1}{18} -a + 2e - f- h \\
 (\alpha_{35}^2)^2 = (\alpha_{67}^2)^2 &= \smallfrac{1}{18} + a \\
 (\alpha_{03}^3)^2 &= \smallfrac{1}{18} - c - e - f + g \\
(\alpha_{17}^3)^2 &= \smallfrac{1}{18} -d - e + f - \half g + h \\
(\alpha_{25}^3)^2 = (\alpha_{46}^3)^2 &= \smallfrac{1}{18} + \half c +
\half d + e - \half g - \half h \\
 (\alpha_{04}^4)^2 &= \smallfrac{1}{18}-b + \half c - \half d - 2 e +
 f + \half g + \half h \\ (\alpha_{12}^4)^2 &= \smallfrac{1}{18} -b -
 \half c + \half d + 2 e - f - \half g - \half h \\ (\alpha_{36}^4)^2
 = (\alpha_{57}^4)^2 &= \smallfrac{1}{18} + b \\
 (\alpha_{05}^5)^2 &= \smallfrac{1}{18} - \half c - \half d + 2e +
\half g - \half h \\ (\alpha_{16}^5)^2 &= \smallfrac{1}{18}- \half c -
\half d - 2e - \half g + \half h \\ (\alpha_{23}^5)^2 &=
\smallfrac{1}{18} + c \\ (\alpha_{47}^5)^2 &= \smallfrac{1}{18} + d \\
  (\alpha_{06}^6)^2 &= \smallfrac{1}{18} + e \\ (\alpha_{15}^6)^2 &=
  \smallfrac{1}{18} - 3e \\ (\alpha_{27}^6)^2 = (\alpha_{34}^6)^2 &=
  \smallfrac{1}{18} + e \\
 (\alpha_{07}^7)^2 &= \smallfrac{1}{18} - f - g- h \\
 (\alpha_{13}^7)^2 &= \smallfrac{1}{18} + f \\ (\alpha_{26}^7)^2 &=
 \smallfrac{1}{18} + g \\ (\alpha_{45}^7)^2 &= \smallfrac{1}{18} + h
 \\
\end{align*}
To see that the metric nilpotent Lie algebras in the family 
defined by these solutions are
not all isomorphic, assume there is 
an isometric isomorphism $\phi$ between two
elements $(\frakn,Q)$ and $(\frakn^\prime,Q^\prime)$ of the family.  The map
$\phi$ must preserve the center $\frakz = \myspan \{Z_i\}_{i=1}^7$
and the orthogonal complement of the center
$\frakv = \myspan \{V_i\}_{i=0}^7.$  The   natural grading automorphism 
that is the identity on $\frakv$ and twice the identity on $\frakz$
defines gradings of   $\frakn$ and $\frakn^\prime.$
We use this grading to apply Theorem  \ref{Ricci-K} 
The set $S$ of values
 $K(X \wedge Y) =  - \frac{3}{4} \| \ad_XY \|^2$
for unit $X, Y \in \frakv,$ is the same for 
 $(\frakn_1,Q_1)$ and $(\frakn_2,Q_2).$  Therefore, the set $S$ is
an isometry invariant.  

Clearly the set $S,$ modulo rescaling, takes on infinitely many values
as the structure constants vary over the set of solutions we have found,
so there is  continuous family of soliton metrics on nilpotent Lie
algebras, pairwise nonhomothetic, defined by the solutions given 
above.    
\end{example}
  In the next example, we consider
a horosphere of maximal dimension for the rank $n-1$ symmetric space
$\calP_n = SL_n(\boldR)/SO(n).$
Let $\frakt_n$ denote the vector space of strictly upper triangular $n
 \times n$ matrices.  We use $\calB = \{ E_{ij} | 1 \le i < j \le n\}$
 as a basis, where $E_{ij}$ is the matrix with a one in the $(i,j)$
 position and zeroes elsewhere, and we let $Q$ be the inner product
 such that $\calB$ is orthonormal.  Let $\alpha_{ijk}$ denote the
 structure constant $\la [E_{ij},E_{jk}] , E_{ik} \ra,$ for $1 \le i <
 j < k \le n.$ The nonzero structure constants are indexed by
 $\Theta_n.$ Let $\mu_0$ be the element of $\Lambda^2 \frakt_n^\ast
 \otimes \frakt_n$ with $\alpha_{ijk} = 1$ for all $(i,j,k)$ in
 $\Theta_n;$ the nilmanifold corresponding to the metric Lie algebra
 $((\frakt_n)_{\mu_0},Q)$ is isometric to a horosphere for a regular
 geodesic in the symmetric space $\calP_n.$
\begin{theorem}\label{tn}
If $n \le 4,$ then $((\frakt_n)_{\mu_0},Q)$ is the only nilsoliton in
its equivalence class $\Omega_{\Lambda}$ up to scaling.  For $n > 4,$
there are continuous families of nonisometric nilsoliton metrics
around $((\frakt_n)_{\mu_0},Q)$ in its equivalence class
$\Omega_{\Lambda}.$
\end{theorem}

\begin{proof}
The dimension of $\frakt_n$ is $\left( \begin{smallmatrix} n \\ 2
\end{smallmatrix} \right), $ and there are    $\left( 
\begin{smallmatrix} n \\ 3
\end{smallmatrix} \right)$ nonzero structure  
constants $\alpha_{ijk}$ (where $(i, j, k)$ is in $\Theta_n$).  It is
easily checked that the Jacobi equation holds for all choices of
structure constants, noting that the only nonzero triple brackets of
basis vectors are of the form $[E_{ij},[E_{jk},E_{kl}]],$ where
$(i,j,k)$ is in $\Theta_5.$

If $n=3,$ then $((\frakt_3)_{\mu_0},Q)$ is the three-dimensional
Heisenberg Lie algebra.  If if $n=4,$ a brief computation shows that
$((\frakt_4)_{\mu_0},Q)$ is the unique nilsoliton metric (up to
scaling) in its equivalence class $\Omega_\Lambda.$

When $n = 5,$ the dimension of $\frakt_5$ is $\left(
\begin{smallmatrix} 5 \\ 2
\end{smallmatrix} \right) = 10$ and $|| \Lambda || = 
\left( \begin{smallmatrix} 5 \\ 3
\end{smallmatrix} \right).$ also ten. There exists a  continuous family
of nilsoliton metrics around $((\frakt_5)_{\mu_0},Q).$ MAPLE shows
that the system $Uv = \onevector{10}$ has general solution
\[ \begin{bmatrix} (\alpha_{123})^2 \\  (\alpha_{124})^2 \\  (\alpha_{125})^2 
\\ (\alpha_{234})^2 \\ (\alpha_{235})^2 \\ (\alpha_{134})^2 \\
(\alpha_{345})^2 \\ (\alpha_{145})^2 \\ (\alpha_{245})^2 \\
(\alpha_{135})^2 \end{bmatrix} = \frac{1}{5}
\begin{bmatrix} 1 \\ 1 \\ 1 \\1 \\ 1\\ 1 \\ 1 \\ 1 \\1  \\ 1 \end{bmatrix} +
s \begin{bmatrix} 1 \\ -1 \\ 0 \\ -1 \\ 0 \\ 1 \\ 0 \\ 0 \\ 0 \\ 0
\end{bmatrix} +
t \begin{bmatrix} 0\\ 1\\ -1 \\ 1\\ -1\\ 0\\ -1\\1 \\0 \\0
\end{bmatrix} + u \begin{bmatrix} 0\\ 0\\ 0\\1 \\ -1\\ 0\\-1 \\0 \\
1\\ 0\end{bmatrix} + v \begin{bmatrix} 1 \\ 0\\ -1\\0 \\ -1\\0 \\ 0\\
0\\ 0\\1 \end{bmatrix} .\] By Theorem \ref{allnilsolitons}, the Ricci
endomorphisms are all the same for all choices of structure constants,
so we need to use sectional curvature to distinguish nonisometric
examples.  To show that $(\frakt_\mu,Q)$ for $\mu$ where $(s, t, u, v)
\ne (0,0,0,0) $ are not isometric to $((\frakt_5)_{\mu_0},Q),$ one
can argue as in Example \ref{cayley},  applying 
Theorem \ref{Ricci-K} using the
canonical grading of $\frakt_5$ as the sum 
 $\bigoplus_{k=1}^{4}\frakg_k,$ where
$\frakg_k = \myspan_{\boldR} \{ E_{ij} \in \calB  :  k = j - i \},$
 for $k=1, \ldots, 4.$ 

If $n > 5,$ then $\left( \begin{smallmatrix} n \\ 2 \end{smallmatrix}
\right) - \left( \begin{smallmatrix} n \\ 3 \end{smallmatrix} \right)
> 2,$ so $U$ is singular by Corollary \ref{m>n}.  The standard $\mu_0$
gives a positive solution to the nilsoliton linear system, so by
continuity, there is a continuous family of positive solutions around
the symmetric solution in its class $\Omega_{\Lambda}.$ As in the $n =
5$ case, the sectional curvature serves as an invariant to distinguish
the perturbed metric Lie algebras from the standard one.   
\end{proof}

\subsection{Regular graphs}
Recall that a graph is {\em regular} if its automorphism group acts
transitively on vertices.  With respect to natural choices of bases,
the nilsoliton metric Lie algebras defined by horospheres of rank one
symmetric spaces of noncompact type have generalized Dynkin diagrams
that are regular graphs. The structure vectors are of the form
$\boldone,$ defining weightings that assign the value one to each
vertex of the generalized Dynkin diagram.  The regularity of the graph
insures that Equation \eqref{weighting} holds for some $\beta$, so that
$\boldone$ is a solution of $U v = -2\beta \boldone.$  In general, we have
 the following corollary to Theorem \ref{mainthm}.
\begin{corollary}\label{regulardynkin}
Let $(\frakn,Q)$ be an inner product space with Ricci eigenvector
basis $\calB.$ If the generalized Dynkin diagram $S(U)$ associated to
a nontrivial level set $\Omega_\Lambda$ in $\Lambda^2 \frakn^\ast
\otimes \frakn$ is a regular graph, then any $\mu$ in $\Omega_\Lambda$
with structure vector $\boldone$ is a nilsoliton metric.
\end{corollary}
Observe that $\boldone$ is a solution to $Uv = -2\beta \boldone$ if
and only if the columns of $U$ all sum to $-2\beta.$

A large class of examples of nilsoliton metric Lie algebras with
regular generalized Dynkin diagrams was presented by E.\ DeLoff
\cite{deloff79}.  A two-step metric nilpotent Lie algebra $(\frakn,Q)$
is called {\em uniform of type $(m,n,r)$} if there exists an
orthonormal basis \[ \calB = \{ X_1, \ldots, X_n, Z_1, \ldots Z_m\} \]
and numbers $r$ and $s$ such that
\begin{enumerate}
\item{$[V_i,V_j] \in \{ 0, \pm Z_1, \ldots, Z_m\},
[V_i,Z_l]=[Z_k,Z_l]=0, $}
\item{If $[V_i,V_j]=\pm[V_i,V_k],$ then $V_j=V_k$}
\item{For every $Z_l$ there exist exactly $r$ disjoint pairs
$\{V_i,V_j\}$ with $[V_i,V_j]=Z_l$}
\item{For every $V_i$ there exist exactly $s$ $V_j$ with $[V_i,V_j]\ne
  0.$}
\end{enumerate}
Notice that the number of nonzero bracket relations is $sn=2rm.$

\begin{example} (\cite{deloff79}, see also \cite{wolter91a}) 
A uniform metric Lie algebra is always nilsoliton.  Properties (1) and
(2) imply that $\la \ad_{X_i}, \ad_{X_j}\ra = 0$ for $i \ne j$ and
$\la J_{Z_k}, J_{Z_l}\ra = 0$ for $k \ne l.$ The basis $\calB$ is then
a Ricci eigenvector basis by Corollary \ref{orthogonal}.  Property 3
implies that $\| J_{Z_l} \| = 2r$ for all $l$ and Property 4 makes $\|
\ad_{X_i}\| = s$ for all $i$ The generalized Dynkin diagram $S(U)$ is
a regular graph with $sn$ vertices and $2s + r$ edges at each vertex,
so by Corollary \ref{regulardynkin}, $(\frakn,Q)$ is nilsoliton.  Then
Theorem \ref{riccitensor} gives that Ricci eigenvalues $\kappa_{X_i}=
- \frac{1}{2}s$ and $\kappa_{Z_l} = \frac{1}{2}r.$ By Theorem
\ref{nilsolitonconstraint}, the nilsoliton constant is $\beta =
-\frac{1}{2}(2s+r).$ By Corollary \ref{m>n}, if the number of nonzero
structure constants $sn=2rm$ is larger than the dimension $m + n$ then
there are continuous families of solutions to 
$Uv = [1].$ Example \ref{quaternionic} is uniform of type
$(m,n,r)=(3,8,4)$ and Example \ref{cayley} is uniform of type
$(m,n,r)=(7,8,4)$; in both cases $|| \Lambda || = rm > m + n = \dim
\frakn$ and there are continuous families of nilsolitons around the
solitons from symmetric spaces.
\end{example}

The next example shows that admitting a positive definite symmetric
 derivation is not a sufficient condition for nilpotent Lie algebra to
 admit a nilsoliton metric.

\begin{theorem}\label{R6} 
 There exists a filiform nilpotent Lie algebra that admits a positive
definite derivation but does not admit a nilsoliton metric.
\end{theorem}

\begin{proof}  Let $\frakn$ be the seven-dimensional Lie
algebra with basis $\calB= \{X_i\}_{i=1}^{7}$ with bracket relations
defined by
\begin{align*}  
[X_1,X_i] & = X_{i + 1} \qquad \text{for $i = 2, \ldots , 6$} \\
[X_2,X_i] & = X_{i + 2} \qquad \text{for $i = 3, 4, 5$}
\end{align*}
This is the algebra $R_6$ from the infinite family of filiform Lie
algebras $R_n$ described in \cite{khakimdjanov89}.  The endomorphism
$A$ defined by $A(X_i) = i X_i$ for $i = 1$ to $7$ is a positive
definite derivation of $\frakn,$ and it is the only one up to scaling
(See \cite{goze-hakimjanov} and Theorem 2.12 of
\cite{khakimdjanov89}).

The proof is by contradiction.  Suppose that $Q$ is an inner product
on $\frakn$ so that $(\frakn,Q)$ is a nilsoliton metric Lie algebra
with nilsoliton constant $\beta$.  Eigenvectors for $D$ are also
eigenvalues for $\Ric = D + \beta \Id.$ The positive definite
derivation $D = \Ric - \beta \Id$ is a multiple of $A$ so the vectors
in the $\calB$ must be eigenvectors for the derivation.

As the multiplicities of the eigenvalues of $A$ are all one, we may
rescale each basis vector in $\calB$ and order the basis that we have
a Ricci eigenvector basis $\calB^\prime = \{ X_i^\prime = a_i X_i
\}_{i = 1}^7$ where $A(X_i) = i X_i$ for $i = 1, \ldots, 7.$ The
structure vector $[\alpha^2]$ depends on the values of $a_1, \ldots,
a_7,$ but the set
\[ \Lambda(\frakn,\calB) = \{ (1,i,i+1) \}_{i = 2}^6 \cup 
\{(2,i,i+2) \}_{i = 3}^5 \] does not, as $a_i \ne 0$ for all $i = 1,
\ldots 7.$ By Theorem \ref{mainthm}, $-\half[\alpha^2]$ is a positive
solution to $U v = \onevector{7}.$ Using MAPLE, we find the general
solution to $U v = \onevector{7}$ is
\begin{align*}
 v & = \smallfrac{1}{5}(0,-1,0,3,2,3,0,0)^T + s v_1 + t v_2, \quad
\text{where} \\ v_1 & = (0,1,0,-1,0,-1,1,0)^T \quad \text{and} \\ v_2
&= (0,1,1,-1,-1,-1,0,1)^T \quad \text{and $s,t \in \boldR$}.
\end{align*}
But the first component of the general solution is always zero, a
contradiction.  Thus, $\frakn$ does not admit a nilsoliton metric.
\end{proof}

We conclude the section with an example of a Einstein solvmanifold of
rank two found using the ideas of this section.
\begin{example}\label{rank2} 
Let $\Lambda = \{ (1,2,3),(1,3,4),(1,4,5),(2,3,5)\}.$ This is the set
$\Lambda(\frakn,\calB)$ for one of the filiform metric Lie algebras
$L_n$ described in Subsection \ref{Lnsubsection}.  We saw in Theorem
\ref{Ln} that letting $\alpha_{12}^3 = \alpha_{14}^5 = \sqrt{3}$ and
$\alpha_{13}^4 = 2$ defines a nilsoliton metric Lie algebra with
$\beta = -11/2$ and Ricci vector $\bfRic = -\half(-10,-3,-1,1,3)^T.$
Then eigenvalue vector of the derivation $D_1 = \Ric - \beta \Id$ with
respect to $\calB$ is
\[ v_{D_1} = \bfRic + \frac{11}{2} \onevector{5} = 
\half(2,9,11,13,15)^T. \] Letting $[A_1, X_i] = D_1(X_i)$ defines a
rank one Einstein solvmanifold $(\fraks_1 = \la A_1 \ra \oplus \frakn
, \widetilde Q_1)$.

   Now notice that $\Lambda$ is a subset of $\Sigma_5$ and the
generalized Dynkin diagram $S(U)$ is a minor graph of type $A_2 \oplus
A_1$ in $S(\Sigma_5),$ which is of type $A_3.$ Hence $(L_5,Q)$ admits
another symmetric derivation with eigenvalues $1,2,3,4,5$ each
of multiplicity one.  The commuting symmetric derivations encoded by
$\bfv_1 = (2,9,11,13,15)^T$ and $\bfv_2 = (1,2,3,4,5)^T$ are in the
derivation algebra of $\frakn.$ Using Gram-Schmidt orthogonalization,
we find a vector $\bfv_2^\prime = (26,-33,-7,19,45)^T$ orthogonal to
$\bfv_1$ so that $\myspan \{\bfv_1, \bfv_2^\prime\} = \myspan \{
\bfv_1, \bfv_2\}.$ Let $D_2$ be the derivation defined by the vector
$\bfv_2^\prime.$ Finally, define the metric solvable extension
$(\fraks_2 = \la A_1, A_2 \ra \oplus \frakn, \widetilde Q)$ by taking
$\{A_1, A_2\} \cup \calB$ as an orthonormal basis and setting
$\ad_{A_1} = D_1$ and $\ad_{A_2} = D_2.$

It can be checked, using formulas for the Ricci curvature of Einstein
solvmanifolds of Iwasawa type in \cite{wolter91a} that $(\fraks_2,Q)$
is an Einstein solvmanifold of rank two with nonpositive sectional
curvature.
\end{example}

In \cite{heberinv}, Heber asked which types of eigenvalue vectors $v_D
= (\mu_1, \ldots, \mu_n)$ occur for the derivations $D = \Ric - \beta
\Id$ defining rank one Einstein extensions of nilsolitons.  We remark
that the question can be rephrased as follows.  The Ricci vector for a
nilsoliton $(\frakn,Q)$ and Ricci eigenvector basis $\calB$ is of the
form $ \bfRic = v_D + \beta \onevector{n}.$ On the other hand, by
Theorem \ref{riccitensor}, $ \bfRic = -\half Y^T [\alpha^2], $ where
$[\alpha^2]$ is the structure vector for $(\frakn,Q)$ with respect to
$\calB.$ Hence, $ -\half Y^T [\alpha^2] = (\mu_1, \ldots, \mu_n)^T +
\beta \onevector{n} .$ Thus, in order to prove that there is a
nilsoliton $(\frakn,Q)$ of eigenvalue type $v_D,$ we would need need
to show that the vector $-2( v_D + \beta \onevector{n})$ is in the
closure of the positive cone spanned by root vectors in $\{ Y_{ij}^k
\, | \, (i,j,k) \in \Sigma_{(\mu_1, \ldots, \mu_n)}\}.$ Notice that
$v_D$ itself is orthogonal to the cone.

\section{Adding nilsolitons}\label{adding}
Now we prove Theorem \ref{addingnilsolitons}.
\begin{proof}
When $\Lambda(\frakn_{\mu_1},\calB)$ and
$\Lambda(\frakn_{\mu_1},\calB)$ are disjoint sets, it can be seen from
the expression for the Ricci form in Theorem \ref{niceformula} that
the nil-Ricci endomorphism is linear with respect to the Lie bracket
defining it.  Since $\calB$ is a Ricci eigenvector basis for
$(\frakn_{\mu_1},Q)$ and $(\frakn_{\mu_2},Q),$ this linearity implies
that $\calB$ is also a Ricci eigenvector basis for the sum
$(\frakn_{\mu_1 + \mu_2},Q),$ and the Ricci vector for $\frakn_{\mu_1
+ \mu_2}$ with respect to $\calB$ is
\[
\bfRic_{\mu_1 + \mu_2}^\calB = \bfRic_{\mu_1}^\calB +
\bfRic_{\mu_2}^\calB.
\] 

Now we consider the quantities $Y_{ij}^k \, \bfRic_{\mu_1 + \mu_2}$ as
$(i,j,k)$ varies over $\Lambda(\frakn_{\mu_1 + \mu_2},\calB).$ If the
triple $(i,j,k)$ is in the set $\Lambda(\frakn_{\mu_1},\calB),$ then
\begin{align*} 
Y_{ij}^k \, \bfRic_{\mu_1 + \mu_2} &= Y_{ij}^k \, (\bfRic_{\mu_1} +
\bfRic_{\mu_2}) \\ &= Y_{ij}^k \, \bfRic_{\mu_1} + Y_{ij}^k \,
\bfRic_{\mu_2} \\ &= -2\beta_1 + c_2
\end{align*}
and similarly, if $(i,j,k)$ is in $\Lambda(\frakn_{\mu_2},\calB),$
then $Y_{ij}^k \, \bfRic_{\mu_1 + \mu_2} = -2\beta_2 + c_1.$ If
$-2\beta_1 + c_2= -2\beta_2 + c_1,$ the quantity $Y_{ij}^k \,
\bfRic_{\mu_1 + \mu_2}$ is constant as $Y_{ij}^k$ varies over
$\Lambda(\frakn_{\mu_1 + \mu_2},\calB).$ By Theorem
\ref{nilsolitonconstraint}, the metric algebra $(\frakn_{\mu_1 +
\mu_2},Q)$ is a nilsoliton with nilsoliton constant $\beta = -2\beta_1
+ c_2.$  
\end{proof}

A trivial application of Theorem \ref{addingnilsolitons} is that the
direct sum of nilsoliton metric Lie algebras with the same nilsoliton
constant is a nilsoliton metric Lie algebra.  More generally, one can
think of defining a new metric nilpotent Lie algebra from an old one
by adding new basis vectors and new bracket relations; sometimes
Theorem \ref{addingnilsolitons} can be used to show that such a
procedure defines a new nilsoliton metric Lie algebra, as we
illustrate in the following example.
\begin{example} Let $(\frakn,Q)$ be a nilsoliton metric Lie algebra 
with nilsoliton constant $\beta$ and Ricci eigenvector basis $\calB.$
Let $Z$ be an element of the Ricci eigenvector basis in the center of
$\frakn$ with positive eigenvalue $\kappa_Z.$ (Such an element exists
because the Ricci endomorphism is positive definite on the center by
Theorem \ref{niceformula}.)  Let $\frakn^\prime = \boldR^2 \oplus
\frakn.$ Fix basis vectors $A$ and $B$ in $\boldR^2$ and endow
$\frakn^\prime$ with the metric which restricts to $Q$ on $\frakn$ and
makes $A$ and $B$ orthonormal to to each other and orthogonal to
$\frakn$.  Define a nilpotent structure on $\frakn^\prime$ by adding
bracket relation $[A,B]= a Z$ to the bracket relations already defined
by the nilpotent structure on $\frakn.$ The Jacobi identity holds for
$\frakn^\prime,$ and $\frakn^\prime$ is nilpotent.  If $a^2 =
\frac{-4\beta - 2\kappa_Z}{5},$ then Theorem \ref{addingnilsolitons}
can be used to show that $(\frakn^\prime,Q^\prime)$ is a nilsoliton
metric Lie algebra.  The Heisenberg algebra of dimension $2n+3$ may be
constructed from the Heisenberg algebra of dimension $2n+1$ using this
method.
\end{example}

For odd $n \ge 3,$ the $n=(2m+1)$-dimensional filiform nilpotent Lie
algebra $Q_n$ is defined relative to the basis $\{X_i\}_{i = 0}^n$ by
\begin{align}
\label{defQna} [X_0,X_i] &= X_{i + 1}, \quad i = 1, \ldots , n-2 \\
\label{defQnb} [X_i,X_{n-i}] &= (-1)^i X_n, \quad i = 1, \ldots, n-1.  
\end{align}
\begin{theorem}\label{Qn}
For all $n \ge 3,$ the Lie algebra $Q_n$ admits a unique nilsoliton
metric.
\end{theorem}

\begin{proof} 
Let $Q$ be the inner product on $\frakn$ with respect to which $\calB$
is orthonormal.  By Corollary \ref{orthogonal}, $\calB$ is a Ricci
eigenvector basis, even if we change the values of the nonzero
structure constants.

The set of nonzero structure constants coming from brackets defined in
 Equation \eqref{defQna} is
\[\Lambda_1 = \{ (0,1,2), (0,2,3), \ldots , (0,n-3,n-1)\}\]
and the set of nonzero structure constants coming from brackets
 defined in Equation \eqref{defQnb} is
\[ \Lambda_2 = \{ (1,n-1,n), (2,n-2,n), \ldots, (m,m+1,n)\}. \] 
Notice that $n$ does not occur in any triple in $\Lambda_1,$ and that
$0$ does not occur in any triple in $\Lambda_2,$ and the two sets of
triples are disjoint.

The bracket relations in Equation {\eqref{defQna}} are the same as
those for the nilpotent Lie algebra $L_{n-1},$ as in Theorem \ref{Ln}.
In the proof of that theorem, we saw that letting $[X_0,X_i] =
\sqrt{i(n-1-i)} \, X_{i + 1}$ defines a nilsoliton metric algebra
$(\frakn_{\mu_1},Q)$ isomorphic to $L_{n-1} \oplus \boldR$ with Ricci
vector
\[ \bfRic_{\mu_1}^\calB = (  -\smallfrac{1}{12}(n-2)(n-1)n,
-\half(n-2), -\half(n-2) + 1, \ldots , \half(n-2),0)^T \] and
nilsoliton constant $\beta_1 = -\smallfrac{1}{12}(n-2)(n-1)n - 1.$

The root vectors $Y_{ij}^k$ for bracket relations in Equation
\eqref{defQnb} are the same as those for the $(2m+1)$-dimensional
Heisenberg algebra $\frakh_m$ as presented in Example
\ref{heisenberg}.  Letting $\alpha_{ij}^k=1$ for all $(i,j,k) \in
\Lambda_2$ defines a nilsoliton metric Lie algebra
$(\frakn_{\mu_2},Q)$ isomorphic to $\boldR \oplus \frakh_m$ with Ricci
vector
\[ \bfRic_{\mu_2}^\calB = -\half(0, 1,1, \ldots, 1,-m)^T\]
and nilsoliton constant $\beta_2 = -\frac{1}{2}(2 + m).$

It can be checked that for all $Y_{ij}^k$ with $(i,j,k)$ in the second
index set $\Lambda_2,$ the product $ Y_{ij}^k \, \bfRic_{\mu_1}^\calB$
is zero, and for all $Y_{ij}^k$ with $(i,j,k)$ in the first index set
$\Lambda_1,$ the product $ Y_{ij}^k \, \bfRic_{\mu_2}^\calB$ is zero.

If we rescale the structure constants for $\frakn_{\mu_1}$ and
$\frakn_{\mu_2}$ by $ \sqrt{- \beta_2}$ and $\sqrt{ - \beta_1}$
respectively, then the rescaled metric nilpotent Lie algebras
$\frakn_{ - \beta_2 \mu_1}$ and $\frakn_{- \beta_1 \mu_2}$ both have
nilsoliton constant $-\beta_1 \beta_2.$ By Theorem \ref{adding}, the
metric algebra $(\frakn= \frakn_{\sqrt{-\beta_2} \mu_1 + \sqrt{-
\beta_1} \mu_2},Q)$ satisfies the nilsoliton condition with nilsoliton
constant $- \beta_1 \beta_2.$

 Using the change of basis
\begin{align*} 
X_0^\prime &= \smallfrac{1}{\sqrt{-\beta_2}}X_0 \\ X_1^\prime &= X_1
 \\ X_i^\prime &= \left( \Pi_{i = 1}^{i - 1} a_i \right) X_i \quad
 \text{for $i= 2$ to $n - 1$} \\ X_n^\prime &= \sqrt{-\beta_1}X_n
 \end{align*} and the symmetry $a_i = a_{n - 1 -i},$ it can be shown
 that $\frakn$ is isomorphic to $Q_n.$

Thus, we have shown that $(\frakn_{-\beta_2 \mu_1 - \beta_1 \mu_2},Q)$
is a nilsoliton metric Lie algebra isomorphic to $Q_n.$ 
Uniqueness follows from \cite{lauret01a}. 
\end{proof}

Given a $k$-step nilpotent Lie algebra $\frakn,$ a graded Lie algebra
$\gr \frakn$ is defined by on the vector space $\gr \frakn = \oplus
_{i = 1}^k (\frakn^{(i - 1)} / \frakn^{(i)} )$ by the bracket
structure
\[ [x +  \frakn^{(i)}, y +  \frakn^{(j)} ] = [x,y] + \frakn^{(i + j)}, 
\qquad x \in \frakn^{(i-1)}, y \in \frakn^{(j-1)}.\] If $\frakn$ is
isomorphic to $\gr \frakn,$ then $\frakn$ is called {\em naturally
graded.}  In \cite{vergne70}, M. Vergne showed that every naturally
graded filiform Lie algebra is isomorphic to $L_n$ or $Q_n.$ So in
Theorems \ref{Ln} and \ref{Qn} we have shown:
\begin{corollary}\label{naturallygraded}
 Every naturally graded filiform Lie algebra admits a Ricci nilsoliton
  metric that is unique up to scaling.
\end{corollary}

\section{Acknowledgments}

The author is grateful to Jorge Lauret, Gerald Payne and 
Wayne Polyzou for helpful comments,  and to
Dennis Stowe for valuable discussions. 

\bibliographystyle{amsalpha}
\bibliography{bibfile}

\end{document}